\documentclass[final]{siamart1116}

 \newtheorem*{example}{Example}
 \numberwithin{equation}{section}
 \newtheorem{remark}[theorem]{remark}
\usepackage{amssymb,verbatim}
\usepackage{amsfonts}

\usepackage{enumerate}
\usepackage[normalem]{ulem}
 \newcommand{\I}{\mathrm{i}}
\newcommand{\D}{\mathrm{d}}

\newcommand{\lb}{\left(}

\newcommand{\vp}{\varphi}

\newcommand{\rb}{\right)}
\newcommand{\PD}{\partial}

\newcommand{\wt}{\widetilde}

\newcommand{\Dc}{\mathcal{D}}

\newcommand{\Fc}{\mathcal{F}}
\newcommand{\Gc}{\mathcal{G}}

\newcommand{\Tc}{\mathcal{T}}

\newcommand{\Beq}{\begin{equation}}
\newcommand{\Eeq}{\end{equation}}
\newcommand{\beq}{\begin{equation*}}
\newcommand{\eeq}{\end{equation*}}
\newcommand{\bal}{\begin{align}}
\newcommand{\eal}{\end{align}}

\newcommand{\g}{\gamma}

\newcommand{\n}{\nabla}

\newcommand{\A}{\alpha}
\newcommand{\B}{\beta}
\newcommand{\bp}{\begin{prob}}
\newcommand{\ep}{\end{prob}}
\newcommand{\bpr}{\begin{proof}}
\newcommand{\epr}{\end{proof}}

\newcommand{\om}{\omega}

\newcommand{\Go}{\mathcal{G}_1}
\newcommand{\Goo}{{\mathcal{G}_1^1}}
\newcommand{\Got}{{\mathcal{G}_1^2}}
\newcommand{\Goth}{{\mathcal{G}_1^3}}

\newcommand{\Goop}{{\mathcal{G}_1^{1+}}}
\newcommand{\Gotp}{{\mathcal{G}_1^{2+}}}

\newcommand{\Gooo}{{\mathcal{G}_1^{1o}}}
\newcommand{\Goto}{{\mathcal{G}_1^{2o}}}

\newcommand{\Goom}{{\mathcal{G}_1^{1-}}}
\newcommand{\Gotm}{{\mathcal{G}_1^{2-}}}

\newcommand{\Gt}{\mathcal{G}_2}

\newcommand{\Gth}{\mathcal{G}_3}
\newcommand{\Gtho}{{\mathcal{G}_3^1}}
\newcommand{\Gtht}{{\mathcal{G}_3^2}}
\newcommand{\Gthth}{{\mathcal{G}_3^3}}
\newcommand{\Gthf}{{\mathcal{G}_3^4}}

\newcommand{\sgp}{\operatorname{\sigma^+}}
\newcommand{\sgo}{\operatorname{\sigma^o}}
\newcommand{\sgm}{\operatorname{\sigma^-}}

\newcommand{\xl}{{x_\ell}}
\newcommand{\xr}{{x_r}}

\newcommand{\CF}{\mathcal{C}_{\Fc}}
\newcommand{\CG}{\mathcal{C}_{\Gc}}

\newcommand{\paren}[1]{ \left({#1}\right)}
\newcommand{\sparen}[1]{ \left\{{#1}\right\}}
\newcommand{\bparen}[1]{ \left[{#1}\right]}

\newcommand{\abs}[1]{ \left|{#1}\right|}
\newcommand{\norm}[1]{ \left\|{#1}\right\|}

\newcommand{\st}{\,:\,}

\newcommand{\smo}{\setminus \mathbf{0}}
\newcommand{\bel}[1]{\begin{equation}\label{#1}}
\newcommand{\be}{\begin{equation}}
\newcommand{\ee}{\end{equation}}

\newcommand{\gr}{\gamma_R}
\newcommand{\gt}{\gamma_T}

\newcommand{\rtwo}{\mathbb{R}^2}
\newcommand{\rr}{\mathbb{R}}
\newcommand{\zz}{\mathbb{Z}}

\newcommand{\eps}{\epsilon}
\newcommand{\supp}{\operatorname{supp}}

\newcommand{\lsp}{\left\{}
\newcommand{\rsp}{\right\}}
\newcommand{\Stwo}{\Sigma_{2}}
\newcommand{\Stwox}{\Sigma_{2,X}}

\newcommand{\so}{s_0}

\newcommand{\Gotc}{{\Gc_1^* \Gc_2}}
\newcommand{\At}{\widetilde{A}}
\newcommand{\Bt}{\widetilde{B}}

\usepackage{lipsum}
\usepackage{amsfonts}
\usepackage{graphicx}
\usepackage{epstopdf}
\usepackage{algorithmic}
\ifpdf
  \DeclareGraphicsExtensions{.eps,.pdf,.png,.jpg}
\else
  \DeclareGraphicsExtensions{.eps}
\fi

\numberwithin{theorem}{section}

\newcommand{\TheTitle}{Singular FIO\MakeLowercase{s} in SAR Imaging} 
\newcommand{\TheAuthors}{Ambartsoumian, Felea, Krishnan, Nolan, and Quinto}

\headers{\TheTitle}{\TheAuthors}

\title{{Singular FIO\MakeLowercase{s} in SAR Imaging, II: Transmitter and
 Receiver at Different Speeds}\thanks{Submitted to the editors DATE.
}}

\author{
  G. Ambartsoumian\thanks{Department of Mathematics, University of Texas at Arlington, TX, USA (\email{gambarts@uta.edu})}
  \and
  R. Felea\thanks{(Corresponding Author) School of Mathematical Sciences, Rochester Institute of Technology, NY, USA (\email{rxfsma@rit.edu})}
  \and
  V. P. Krishnan\thanks{TIFR Centre for Applicable Mathematics, Bangalore, Karnataka, India
 (\email{vkrishnan@math.tifrbng.res.in})}
\and 
C. J. Nolan\thanks{Department of Mathematics and Statistics, University of Limerick, Ireland (\email{clifford.nolan@ul.ie})}
\and
 E. T. Quinto\thanks{Department of Mathematics,
Tufts University, Medford, MA, USA (\email{todd.quinto@tufts.edu})
}
}
\usepackage{amsopn}

\ifpdf
\hypersetup{
  pdftitle={\TheTitle},
  pdfauthor={\TheAuthors}
}
\fi

\begin{document}

\maketitle

% REQUIRED
\begin{abstract}
In this article, we consider two bistatic cases arising in synthetic
aperture radar imaging: when the transmitter and receiver are both
moving with different speeds along a single line parallel to the
ground in the same direction or in the opposite directions. In both
cases, we classify the forward operator $\Fc$ as a Fourier integral
operator with fold/blowdown singularities.  Next we analyze the normal
operator $\Fc^*\Fc$ in both cases (where $\Fc^{*}$ is the
$L^{2}$-adjoint of $\Fc$).  When the transmitter and receiver move in the same
direction, we prove that $\Fc^*\Fc$ belongs to a class of
operators associated to two cleanly intersecting Lagrangians,
$I^{p,l} (\Delta, C_1)$.  When they move in opposite directions,
$\Fc^*\Fc$ is a sum of such operators. In both cases artifacts appear
and we show that they are, in general, as strong as the
bona-fide part of the image.  Moreover, we demonstrate that as soon as
the source and receiver start to move in opposite directions, there is
an interesting bifurcation in the type of artifact that appears in the
image.
\end{abstract}

% REQUIRED

% REQUIRED

\section{Introduction}

\par  Synthetic Aperture Radar (SAR) is a high-resolution
imaging technology that uses antennas on moving platforms to send
electromagnetic waves to objects of interest and measures the
scattered echoes. These are then processed to form an image of the objects. For a good overview of SAR imaging, especially from a mathematical point of view, we refer the reader to \cite{Cheney2001,CheneyBordenbook}. In monostatic SAR imaging, the moving transmitter also acts as a receiver, whereas in bistatic SAR imaging, the transmitter and receiver are located on different platforms. %The collected data are then used to generate an
%image of \tred{objects on the ground.}
%that object.
  % The measured
%signals are then used to produce an image of the objects %\cite{NC2004,
%Symes2009}.

%\tc{We now have only 15 references (before the referee's complaint, we
%had 27 references).  I added \cite{Symes2009}, but are
%there one or two other general SAR references we should add?}

%\vc{I reread the comments of the second referee. (S)he wants more focused references and was not too concerned  about the number of references, if I understand it correctly. I have rewritten that part below; see the first para on the next page.}

Our focus in this article is on a bistatic SAR imaging setup,  where the
transmitter and receiver move along a straight line parallel to the
ground, in the same direction or in the opposite directions, and with
different speeds (see \eqref{def:GtGr}).  Here and in the rest of the
article we assume that the ground is represented by a plane. The SAR
imaging task is then mathematically equivalent to recovering the
ground reflectivity function $V$ from the measured data $\Fc
V$ for a certain period of time along each point of the receiver trajectory.  Since exact reconstruction of $V$ from $\Fc V$ is an extremely
difficult task, a reasonable compromise (acceptable in practice for
most applications) is to find the singularities of $V$ from $\Fc V$.
In particular, this will allow to see the edges (and hence shapes) of
the objects on the ground. Unfortunately even that simpler task may
not be completely accomplishable, since in certain setups $\Fc V$ may
not have enough information for correct recovery of singularities of
$V$. In these cases, the best possible reconstruction of $V$ may miss
certain parts of the original singularities, or have added ``fake''
singularities, called \textit{artifacts}.

In this article, the reconstructed images, including artifacts, are
analyzed using the calculus of singular Fourier integral operators
(FIOs). The forward operator $\Fc$ which maps singularities in the
scene to those in the data is an FIO. It is conventional to
reconstruct the image of an object by using the backprojection
operator, $\Fc^*$ applied to the data $\Fc V$. We study the
normal operator $\Fc^*\Fc$ and the artifacts which appear by using
this method.

The current article is a continuation of our prior work
\cite{AFKNQ:common_midpoint}, where, motivated by certain multiple
scattering scenarios, we considered the case when the transmitter and
receiver move at equal speeds away from a common midpoint along a
straight line. The main result of that article made precise the added
singularities and their strengths (in comparison to the true
singularities) when reconstruction is done using the backprojection
method mentioned above. We showed that the backprojection method
introduces three additional singularities for each true singularity
with potentially no way of avoiding them if the transmitter and
receiver are assumed omnidirectional. One of our main motivations in
studying the case of different speeds for the transmitter and receiver
was to remove some of these artifacts. However, when the transmitter
and receiver move away from each other, the backprojection method
still introduces additional artifacts, which in the limiting case
(when the speeds are equal) gives the artifacts considered in
\cite{AFKNQ:common_midpoint}.

The microlocal analysis of the normal operator in the study of generalized Radon transforms and in imaging problems  has a long history. Guillemin and Sternberg were the first to study integral geometry problems from the FIO and microlocal analysis point of view, and made fundamental contributions \cite{GuilleminSternberg, Gu1985}. Later, paired Lagrangian calculus introduced by Melrose-Uhlmann \cite{MU} and Guillemin-Uhlmann \cite{Guillemin-Uhlmann}, and also studied in  Antoniano-Uhlmann \cite{Antoniano-Uhlmann} was used by Greenleaf-Uhlmann in several of their highly influential works on the study of generalized Radon transforms \cite{GU1989,GU1990a}. Microlocal techniques have also been very useful in the context of seismic imaging \cite{Beylkin-1985, NolSym1997,tK-Smit-Verdel, Nolan-fold_caustics,Stolk_deHoop_CPAM,Dehoop-InsideOut, Felea-Greenleaf-MRL}),  in sonar imaging see  \cite{FG, Felea-Greenleaf-Pramanik, QRS2011}), in X-ray Tomography; in addition to works mentioned above also see
\cite{Quinto93, Katsevich:1997vo,FLU,Felea-Q2011, Frikel:2013gb}), and in tensor tomography \cite{SU1,SU2, Uhlmann-Vasy}.

The microlocal analysis of linearized SAR imaging operators (both monostatic and bistatic) was done in \cite{NC2004, YCY, RF1,AFKNQ:common_midpoint, Stefanov-Uhlmann-SAR}. Bistatic SAR imaging problems, due to the fact that the transmitter and receiver are spatially separated, naturally lead to the study of elliptical Radon transforms, which are also of independent interest. These have been studied in the literature as well \cite{Amb-Kri,ABKQ,Moon-Heo}.
% and in harmonic analysis \cite{Greenleaf-Seeger1994, Greenleaf-Seeger1998}.

 %and  Previous work in SAR using the composition calculus was
%done in \cite{NC2004, RF1, Krishnan-Quinto, AFKNQ:common_midpoint,Felea-Nolan}.
%There is extensive prior literature involving the use
%$of microlocal techniques in the study of generalized Radon transforms,
%integral geometry, scattering theory and harmonic analysis
%\cite{GuilleminSternberg,GU1985,
%Melrose-Taylor,Guillemin-Cosmology-book, GU1989,GU1990a, GU1990b,
%Greenleaf-Seeger1994, Greenleaf-Seeger1998, NC2004, SU1,SU2,SU3,RF1,
%RF2, Felea-Greenleaf, Felea-Greenleaf-Pramanik,
%Krishnan-Quinto,Uhlmann-Vasy,Felea-Nolan,Stefanov-Uhlmann-Vasy}.

The article is organized as follows: In Section
\ref{sect:main results} we state the main facts and results: the
positions of the transmitter and receiver that we consider
\eqref{def:GtGr}, the forward operator $\Fc$ \eqref{def:F}, the
canonical relation of $\Fc$ and the properties of the
projections from its canonical relation ($\pi_L$ and $\pi_R$)
(Theorems \ref{thm:F alpha>=0} and \ref{thm:G alpha<0}).  Then,
we describe the composition calculus results (Theorems \ref{thm:F*F
alpha>=0} and \ref{thm:G*G}).  \par In Section
\ref{sect:preliminaries} we recall briefly the definition of the
fold/blowdown singularities and the properties of the $I^{p,l}$
classes we need in this article, and Section
\ref{sect:common_midpoint} briefly summarizes the main result of
\cite{AFKNQ:common_midpoint}. \par In Section \ref{sect:F} we consider
the case when the transmitter and receiver are moving in the same
direction along a line parallel to the ground. We show that the
normal operator $\Fc^{*}\Fc$ has a distribution kernel belonging to
the paired Lagrangian distribution class $I^{2m,0}(\Delta, C_1)$ where
$\Delta$ is the diagonal relation and $C_1$ is the graph of a simple
reflection map about the $x$-axis. This result is valid even if the
transmitter is stationary, for example, when the transmitter is a
fixed radio tower and the receiver is a drone.

 In Section \ref{sect:alpha-} we consider the case when the
transmitter and receiver move in opposite directions along the
line, and the analysis becomes considerably more complicated.
To distinguish this case from the case in Section \ref{sect:F},
we denote the forward operator by $\Gc$. First of all, the
projections drop rank by one along two smooth disjoint hypersurfaces
$\Sigma_1 \cup \Sigma_2$.  Then, Theorem \ref{thm:G*G} shows that the
backprojection adds \emph{two} sets of artifacts and the normal
operator $\Gc^*\Gc$ is a sum of operators belonging to $I^{p,l}$
classes: $I^{2m,0}(\Delta, C_1) + I^{2m,0}(\Delta, C_2) +
I^{2m,0}(C_1, C_2)$, where $C_1$ causes the same artifact which
appears for $\A\geq 0$ in Section \ref{sect:F} and $C_2$ is a
two-sided fold (Def.\ \ref{def:fold}).  Finally, in section
\ref{sect:spotlighting} we consider spotlighting, in which certain
portions of the ground are selectively illuminated.  In this case, we
show that the normal operator $\Gc^*\Gc$ belongs to $I^{2m,0} (\Delta,
C_2)$ where $C_2$ is a two-sided fold canonical relation (see Theorem
\ref{thm:beamforming}).

 In Appendix \ref{proofs:iterated regularity}, we prove Theorem
\ref{thm:F*F alpha>=0} using the iterated regularity method and in
Appendix \ref{min-max-times}, we give a geometric explanation of the
points we cannot image in the case considered in Section
\ref{sect:alpha-}.

 In all these situations, we show that additional artifacts (coming
from $C_1$ and $C_2$) could be just as strong as the bona-fide part of
the image, in other words, singularities related to $\Delta$.  In this
article, for $\Fc^*\Fc \in I^{p,l}(\Delta, C)$, the strength of the
artifact ($C$) means the order of $\Fc^*\Fc$ on $C \setminus \Delta$.
We find the order of $\Fc^*\Fc$ on both $\Delta \setminus C$ and $C
\setminus \Delta$ and, if they are the same, then we conclude that in
general, the artifact is as strong (see Remarks \ref{remark:strength
alpha+} and \ref{remark:alpha- general}).

An obvious but perhaps important observation that follows from
the results of this paper is that if one has a choice of having the
source and receiver platforms moving in the same or opposite
directions  along the same straight line, then it is highly preferable to have them move in the same
direction in order to avoid additional set of artifacts,  other than the usual left-right ambiguities.

\section{Statements of the main results} \label{sect:main results}

\subsection{The linearized scattering model} \label{sect:prel-mod} For
simplicity, we assume that both the transmitter and receiver are at
the same height $h>0$ above the ground at all times and that the
transmitter and receiver move at constant but different speeds along
a line parallel to the $x$ axis.  Let
\bel{def:GtGr}\gamma_{T}(s)=(\alpha s,0,h)\qquad \gamma_{R}(s)=(s,0,h)
\ee for $s\in (0,\infty)$ be the trajectories of the transmitter and
receiver respectively.

The case $\alpha= -1$ corresponds to the common midpoint problem,
which was fully analyzed in \cite{AFKNQ:common_midpoint}. Therefore we
will assume $\alpha \neq -1$. We also assume $\alpha \neq 1$, since
 $\alpha=1$ corresponds to the monostatic case
 (where the same device serves as both a transmitter and a
receiver) and has also been fully analyzed in earlier works
\cite{NC2004, RF1}.

We are aware that there
are other cases for the transmitter and receiver to be considered,
like moving along parallel lines at different heights or along skew
lines at different speeds or along intersecting lines in a plane
parallel to the ground.  At this point we can only say that in those
cases the left-right ambiguity which appears in the case considered in
this article will, in general, be lost. However, we will limit
the analysis of $\Fc$ and $\Fc^*\Fc$ only to the case mentioned in
\eqref{def:GtGr} since it is already leads to interesting analysis. We
point out that, in practice, the flight paths can be more complicated
because of turbulence and other factors.

The linearized model for the scattered signal we will use in this
article is \[  \int e^{-\I
\omega\left(t-\frac{1}{c_{0}} R(s,x)\right)}a_0(s,x,\omega)V(x)
\D x \D \omega\] for $(s,t)\in (0,\infty)\times (0,\infty)$, where
$V(x)=V(x_{1},x_{2})$ is the function modeling the object on the
ground, and
\[R(s,x)=\norm{\gamma_{T}(s)-x}+\norm{x-\gamma_{R}(s)}
\] is the bistatic
distance--the sum of the distance from the transmitter to the
scatterer and from the scatterer to the receiver, $c_{0}$ is the speed
of electromagnetic wave in free-space and the amplitude term $a_0$ is
given by
\[
a_0(s,x,\omega)=\frac{\omega^{2}p(\omega)}{16\pi^{2}\norm{\gamma_{T}(s)-x}\norm{\gamma_{R}(s)-x}}.
\]
This function includes terms that take into account the transmitted
waveform and geometric spreading factors.

From now on, we denote the $(s,t)$ space by $Y=(0,\infty)^2$ and the
$(x_1,x_2)$ space by $X=\rtwo$.

For simplicity, we will assume that $c_{0}=1$.  Because the
ellipsoidal wavefronts do not meet the ground for \[t<
\sqrt{(\A-1)^2s^2+4h^2},\] there is no signal for such $t$.  As we
will see, our method cannot image the point on the ground directly
``between'' the transmitter and receiver (see the proof of Theorem
\ref{thm:F alpha>=0} in Section \ref{sect:F}).  Given transmitter and
receiver positions $\A s$ and $s$ respectively, such a point on the
ground has coordinates $\left(\frac{(\A+1)s}{2},0\right)$. Note that
this point on the $x$-axis corresponds to $t=\sqrt{(\A-1)^2s^2+4h^2}$.
For these two reasons, we multiply $a_{0}$ by a cutoff function $f$
that is zero in a neighborhood of \[\sparen{ (s,t)\st s>0,\ 0<t\leq
\sqrt{(\A-1)^{2}s^{2}+4h^{2}}}.\]

In addition, to be able to compose our forward
operator and its adjoint, we further assume that $f$ is compactly
supported and equal to $1$ in a neighborhood of a suitably large
compact subset of
\[
\{(s,t):s>0,\sqrt{(\A-1)^{2}s^{2}+4h^{2}}<t<\infty\}.
\]
 We let $f\cdot a_0=a$, and this gives us the data \bel{def:F}\Fc
V(s,t):=\int
e^{-i\omega\left(t-\norm{x-\gt(s)}-\norm{x-\gr(s)}\right)}
a(s,t,x,\omega) V(x) \D x \D \omega.\ee We require additional cutoffs
for our analysis to work for the case of $\alpha<0$ (see Remarks
\ref{S1S2cutoff} and \ref{rem:g(s,t)}).

Throughout the article we use the following notation
\bel{def:AB}\begin{aligned}
A&=A(s,x)=\norm{x-\gt(s)}=\sqrt{(x_1-\alpha s)^2+x_2^2+h^2}\\
B&=B(s,x) =
\norm{x-\gr(s)}=\sqrt{(x_1-s)^2+x_2^2+h^2}.\end{aligned}\ee
and we define the ellipse
\bel{def:E} E(s,t) = \sparen{x\in \rtwo\st A(s,x)+B(s,x)= t}\ee

We assume that the amplitude function $a \in S^{2}$, that is, it
satisfies the following estimate: For every compact $K \subset Y
\times X$ and for every non-negative integer $\delta$ and for every
$2$-index $\beta=(\beta_{1},\beta_{2})$ and $\lambda$, there is a
constant $c$ such that
\begin{equation}\label{amplitude estimate}
|\partial_{\omega}^{\delta}\partial_{s}^{\beta_{1}}\partial_{t}^{\beta_{2}}\partial_{x}^{\lambda}a(s,t,x,\omega)|\leq
c(1+|\omega|)^{2-\delta}.
\end{equation}
This assumption is satisfied if the transmitted waveform from the
antenna is approximately a Dirac delta distribution.
The  qualitative features
predicted by the approach based on microlocal analysis are consistent with practical reconstructions,
including for example, the well-known right-left ambiguity artifact in
low-frequency SAR images \cite{CheneyBordenbook}.

\subsection{Transmitter and receiver moving in the same direction:
$\A\geq 0$}\label{sect:results alpha+}

 The case $\A\geq 0$ corresponds to the situation when the transmitter
and receiver are traveling in the same direction along a line parallel
to the ground or when the transmitter is stationary ($\alpha = 0$) on
that line. For $\A\geq 0$, we refer to the forward operator by $\Fc$.
We show that for the case $\A\geq 0$, the operator $\Fc$ in
\eqref{def:F} is a FIO of order $\frac{3}{2}$ and study the properties
of the natural projection maps from the canonical relation of $\Fc$.
We have the following results.

\vskip .5 cm
 \begin{theorem}\label{thm:F alpha>=0} Let $\Fc$ be the operator in
 \eqref{def:F} for $\alpha\geq 0$.
 \begin{enumerate}
 \item $\Fc$ is an FIO of order $3/2$.
 \item The canonical relation $\CF\subset T^*Y
\smo \times T^*X\smo$  associated to $\Fc$ is given by
\bel{def:Canonical relation}\begin{aligned}
\CF =\Bigg\{&\Big( s,t, -\omega\left(\frac{x_1-
\alpha s}{A} \alpha+\frac{x_1-s}{B}\right), \omega;\\
&x_1,x_2, \omega\left(\frac{x_1- \alpha s}{A}+\frac{x_1-s}{B}\right),
\omega
\left(\frac{x_2}{A}+\frac{x_2}{B}\right)\Big)\\
&\st \om\neq 0,\ t=A+B
\Bigg\}.
\end{aligned}\ee where $A=A(s,x)$ and $B=B(s,x)$ are defined in
\eqref{def:AB}.  Furthermore,
$(s,x_{1},x_{2},\omega)$ is a global parameterization of $\CF$.
\item Denote the left and right projections from $\CF$
to $T^{*}Y \smo$ and
$T^{*}X \smo$  by $\pi_{L}$ and $\pi_{R}$ respectively. Then $\pi_{L}$ and $\pi_{R}$ drop
rank simply by one on the set
\bel{def:Sigma1}\Sigma_1=\left\{ (s,x_{1},x_{2},\omega) \in \CF \st x_{2}=0 \right\}.\ee
  \item $ \pi_L$ has a fold singularity along
$\Sigma_1$ and $ \pi_{R}$ has a blowdown singularity along
$\Sigma_1$
(see Def.\
\ref{def:fold-blowdown}).
\end{enumerate}
\end{theorem}

We next analyze the imaging operator $\Fc^{*} \Fc $.

 \begin{theorem}\label{thm:F*F alpha>=0} Let $\Fc$ be as in Theorem
\ref{thm:F alpha>=0}. Then $\Fc^*\Fc \in I^{3,0}(\Delta, C_1)$, where
\bel{def:C1}C_1=\{(x_1, x_2, \xi_1, \xi_2; x_1, -x_2, \xi_1,
-\xi_2)\st (x,\xi)\in T^{*}X \smo \} \ee which is the graph of
$\chi_1(x,\xi)=(x_1, -x_2, \xi_1, -\xi_2)$.
\end{theorem}

  \begin{remark}\label{remark:strength alpha+} Since $\Fc^*\Fc \in
I^{3,0} (\Delta, C_1)$, using the properties of the $I^{p,l}$ classes
\cite{Guillemin-Uhlmann}, we have that microlocally away from $C_1$,
$\Fc^*\Fc$ is in $I^{3} (\Delta \setminus C_1)$ and microlocally away
from $\Delta$, $\Fc^*\Fc \in I^{3} (C_1 \setminus \Delta)$.  This
means that $\Fc^*\Fc$ has the same order on both $\Delta\setminus C_1$
and $C_1\setminus \Delta$ which implies that the artifacts caused by
$C_1\setminus \Delta$ will, in general, have the same order as
the reconstruction of the singularities in $V$ that cause them
(see the comments below Def.\ \ref{def:GU}). However, more complicated
behavior can occur including smoothing or cancellation of artifacts.
\end{remark}

\subsection{Transmitter and receiver moving in opposite directions:
$\A<0$}\label{sect:results alpha-}

When $\A<0$, the transmitter and receiver travel away from each other,
and we refer to the forward operator by $\Gc$.

\subsubsection{Further preliminary modifications of the scattered
data} In the case when $\A<0$, we further modify the operator $\Fc V$
considered in Section \ref{sect:prel-mod}.

Our method cannot image a neighborhood of two points on
the ground for a given transmitter and receiver positions in
addition to the points muted by the cutoff function $f$ in Section \ref{sect:prel-mod}. Therefore we modify or
pre-process the receiver data further such that the contribution to it
from a neighborhood of these two points is $0$.
The two points on the
$x_{1}$-axis that we would like to avoid are of the form $\lb
x_{1}^{\pm},0\rb $, where
\begin{align}
& x_{1}^{+}=\frac{2\A s}{\A+1}+\sqrt{-\A
\frac{(\A-1)^{2}}{(\A+1)^{2}}s^{2}-h^{2}},\label{Avoided x1
values +} \\
& x_{1}^{-}=\frac{2\A s}{\A+1}-\sqrt{-\A
\frac{(\A-1)^{2}}{(\A+1)^{2}}s^{2}-h^{2}}\label{Avoided x1 values -}
\end{align} as explained in Remark \ref{S1S2cutoff}.
We define a smooth mute function $g(s,t)$ that is identically $0$ if
the ellipse $E(s,t)$ is near one of the points $(x_1^\pm,0)$; for each
$s$, the corresponding values of $t$ are \bel{def:tpm}
t^{\pm}_{s}=A\paren{s,(x_{1}^{\pm},0)}+B\paren{s,(x_{1}^{\pm},0)} \ee
where $A$ and $B$ are given by \eqref{def:AB}. The points $t^{\pm}_s$
are given explicitly in Appendix \ref{min-max-times}.

With the function $g$, we modify $\Fc$ in \eqref{def:F} by replacing
$a$ by $g\cdot a$ and call it $a$ again. Throughout this section,
corresponding to the case $\A<0$, we will designate the operator as
$\Gc$. That is, we have \Beq\label{def:G} \Gc V(s,t):=\int
e^{-i\omega\left(t-\norm{x-\gt(s)}-\norm{x-\gr(s)}\right)}
a(s,t,x,\omega) V(x) \D x \D \omega, \Eeq where $a$ takes into account
the cutoff functions $f$ in Section \ref{sect:prel-mod} and the
function $g$ defined in the last paragraph.

\begin{theorem}\label{thm:G alpha<0} Let $\Gc
$ be the operator given in \eqref{def:G} for $\alpha<0$. Then
\begin{enumerate}
\item $\Gc$ is an FIO of order $\frac{3}{2}$ \item The canonical
relation $\CG$ associated to $\Gc$ is given by \eqref{def:Canonical
relation} with global parameterization $(s,x_{1},x_{2},\omega)$.

\item The left and right projections $\pi_{L}$ and $\pi_{R}$
respectively from $\CG$ drop rank simply by one on the set $\Sigma
=\Sigma_{1}\cup \Sigma_{2}$ where $\Sigma_1$ is given by
\eqref{def:Sigma1} and
\begin{align}\Sigma_2&=\left\{ (s,x, \omega) \in \CG \st
\frac{\alpha}{A^{2}}+ \frac{1}{B^{2}}=0, \ x_2 \neq 0
\right\}\label{def:Sigma2}\\
& =\Bigg{\{}(s,x, \omega) \in \CG \st \left( x_1- \frac{2 \alpha
s}{\alpha +1}\right)^2+x_2^2= -\alpha s^2\frac{(\alpha-1)^2}{(\alpha
+1)^2}-h^2, \ x_2 \neq 0 \Bigg{\}} \label{Sigma2-circle}
\end{align}
\item $ \pi_L$ has a fold singularity along $\Sigma$ (see Def.\
\ref{def:fold-blowdown}).  \item $ \pi_{R}$ has a blowdown
singularity along $\Sigma_1$ and a fold singularity along $\Sigma_2$
(see Def.\ \ref{def:fold-blowdown}).  \end{enumerate} \end{theorem}

 For convenience, we denote, for each $s$, the projection of the part
of $\Sigma_2$ above $s$ to $\rtwo$ (the projection to the
base of $\pi_R\paren{{\Sigma_2}{\big|_{s}}}$) by $\Stwox(s)$, and
this is the circle described in \eqref{Sigma2-circle} and in Appendix
\ref{min-max-times}.  It can be
written \bel{def:Stwox}\Stwox(s) = \lsp x \st
\frac{\alpha}{A^{2}(s,x)}+ \frac{1}{B^{2}(s,x)}=0, \ x_2 \neq 0
\rsp.\ee

\begin{remark}\label{S1S2cutoff} From Equation \eqref{Sigma2-circle} we
have that
 $\Stwox(s)$ is a
circle of radius $\sqrt{-\alpha s^2\frac{(\alpha-1)^2}{(\alpha
+1)^2}-h^2}$ and centered at $(2\alpha s/(\alpha+1),0)$.

Now we can explain why we need to cutoff the data for ellipses near
the two points given by \eqref{Avoided x1 values +}-\eqref{Avoided x1
values -}.  Since $\pi_R(\Sigma_1)$ intersects $\pi_R(\Sigma_2)$ above
these two points, $\pi_R$ drops rank by two above these points.  So,
we mute data near $(s,t_s^\pm)$ given by \eqref{def:tpm}.  We will
precisely describe this mute, $g$ in Remark \ref{rem:g(s,t)}.
\end{remark}

We now analyze the imaging operator $\Gc^{*}\Gc$. Unlike the case
$\A\geq 0$, this case is more complicated and we consider several
restricted transforms.

\begin{theorem}\label{thm:G*G} Let $ \A \leq 0$ and $\A\neq -1$.  Let $\Gc$ be the
operator in \eqref{def:G} and let
\bel{def:s0-alpha}\so=\frac{h(\A+1)}{\sqrt{-\A}(\A-1)}\ee Then the
following hold:
\begin{enumerate}
\item\label{O1} Let $O_1= \{(s,t): 0<s<\so$ and $0<t<\infty\}$ and let
 $r_{1}$ be a smooth cutoff function that is
compactly  supported in $O_1$.  Consider the
operator $\Gc$ defined in \eqref{def:G} with the amplitude function
$a$ replaced by $r_{1}\cdot a$. Then $\Gc^{*}\Gc\in
I^{3,0}(\Delta,C_{1})$ where $C_{1}$ is defined in \eqref{def:C1}.

\item\label{O2} Let $O_2=\{(s,t): \so<s<\infty$ and
$t_{s}^{-}<t<t_{s}^{+}\}$ where $t_{s}^{\pm}$ is defined in
\eqref{def:tpm}.  Let $r_{2}$ be a smooth cutoff function and
compactly supported in $O_2$. Consider the operator $\Gc$ defined in
\eqref{def:G} with the amplitude function $a$ replaced by $r_{2}\cdot
a$. Then $\Gc^{*}\Gc\in I^{3,0}(\Delta,C_{1})+ I^{3,0}(\Delta,C_{2})+
I^{3,0}(C_{1},C_{2})$ where $C_{2}$ is a two-sided fold given by
\eqref{def:C2}.

\item\label{O3} Let $O_3= \{(s,t): \so<s<\infty$ and $t<t_{s}^{-}$ or
$t>t_{s}^{+}\}$ with $t_{s}^{\pm}$ defined in \eqref{def:tpm}.  Let
$r_{3}$ be a smooth cutoff function compactly supported in $O_3$.
Consider the operator $\Gc$ defined in \eqref{def:G} with the
amplitude function $a$ replaced by $r_{3}\cdot a$. Then $\Gc^{*}\Gc\in
I^{3,0}(\Delta,C_{1})$.
\end{enumerate}
\end{theorem}

\begin{remark}\label{remark:alpha- general} Using the
properties of the $I^{p,l}$ classes for case \ref{O2} of the
theorem, \[\Gc^*\Gc \in I^{3,0} (\Delta, C_1)+I^{3,0}(\Delta, C_2)+
I^{3,0}(C_{1}, C_2)\] implies that artifacts in the reconstruction
could show up because of $C_1$ (reflection in the $x_1$ axis) and
because of $C_2$ (a 2-sided fold). Furthermore, from the discussion
below Definition \ref{def:Ipl}, we have that $\Gc^*\Gc \in I^3(C_i
\setminus \Delta), \ i=1, 2$, $\Gc^*\Gc \in I^3(C_i \setminus C_j), \
i, j=1, 2, i \neq j$ and $\Gc^*\Gc \in I^3(\Delta \setminus C_i), \
i=1, 2$ and thus these artifacts could be, in general, as strong as
the bona-fide part of the image (corresponding to $I^3(\Delta)$).
\end{remark}

\section{Preliminaries: Singularities and $I^{p,l}$
classes}\label{sect:preliminaries}

Here we introduce the classes of distributions and singular FIO we
will use to analyze the forward operators $\Fc$ and $\Gc$ and the
normal operators $\Fc^* \Fc$ and $\Gc^* \Gc$.

 \begin{definition} \cite{GU1990b}\label{def:fold-blowdown} Let $M$ and $N$ be manifolds of
dimension $n$ and let $f:M \to N$ be $C^\infty$. Define $\Sigma= \{m
\in M \st \det(df)_{m} = 0\}$.

\begin{enumerate}
\item $f$ drops rank by one \emph{simply} on $\Sigma$ if for each $m_{0}\in
\Sigma$,\hfil\newline rank~$(df)_{m_{0}}=n-1$ and
$d(\det(df)_{m_{0}})\neq 0$.

\item $f$ has a \emph{Whitney fold} along $\Sigma$ if $f$ is a local
diffeomorphism away from $\Sigma$ and $f$ drops rank by one simply on
$\Sigma$, so that $\Sigma$ is a smooth hypersurface and $\ker \:(d
f)_{m_{0}} \not \subset T_{m_{0}}\Sigma$ for every $m_{0}\in \Sigma$.

\item $f$ is a \emph{blowdown} along $\Sigma$ if $f$ is a local
diffeomorphism away from $\Sigma$ and $f$ drops rank by one simply on
$\Sigma$, so that $\Sigma$ is a smooth hypersurface and $\ker
(df)_{m_{0}} \subset T_{m_{0}}(\Sigma)$ for every $m_{0}\in \Sigma$.
\end{enumerate}
\end{definition}

% \begin{definition}A submanifold $M \subset T^*X$ is {\it non-radial} if
% $\rho \not \in (TM)^{\perp}$, where $\rho=\sum \xi_i
% \partial_{\xi_i}$.
% \end{definition}

 \begin{definition}[{\cite{Melrose-Taylor}}]\label{def:fold} A smooth
canonical relation $C$ for which both projections $\pi_L$ and $\pi_R$
have only (Whitney) fold singularities, is called a \emph{two-sided
fold} or a \emph{folding canonical relation}.
\end{definition}

\par This notion was first introduced by Melrose and Taylor
\cite{Melrose-Taylor}, who
showed the existence of a normal form in $T^*{\rr}^n\times T^*{\rr}^n$.

 \begin{theorem}[{\cite{Melrose-Taylor}}]\label{thm:Melrose-Taylor
 normal form} If dim \ $X=n$ dim \ $Y=n$
and $C \subset (T^*X \smo)\times (T^*Y \smo)$ is a
two-sided fold, then microlocally there are homogeneous canonical
transformations, $\chi_1 :T^*X \to T^*{\rr}^n$ and $\chi_2: T^*Y \to
T^*{\rr}^n$, such that $(\chi_1 \times \chi_2)(C) \subseteq C_0,$
near $ \xi_2 \neq 0$ where, $C_0= N^*\{x_2-y_2=(x_1-y_1)^3; \ x_i=y_i,
\ 3 \le i \le n \}$.
\end{theorem}

We now define $I^{p,l}$ classes. They were first introduced by Melrose
and Uhlmann \cite{MU}, Guillemin and Uhlmann \cite{Guillemin-Uhlmann}
and Greenleaf and Uhlmann \cite{GU1990a} and they have been used in the study of SAR imaging \cite{NC2004,RF1,Krishnan-Quinto, AFKNQ:common_midpoint}.

 \begin{definition}\label{def:cleanly} Two submanifolds $M$ and $N$
intersect {\it cleanly} if $M \cap N$ is a smooth submanifold and $T(M
\cap N)=TM \cap TN$.
\end{definition}

Consider two spaces $X$ and $Y$ and let $\Lambda_0$ and $\Lambda_1$
and $\tilde{\Lambda}_0$ and $\tilde{\Lambda}_1$ be Lagrangian
submanifolds of the product space $T^*X \times T^*Y$.  If they
intersect cleanly, $(\tilde{\Lambda}_0, \tilde{\Lambda}_1)$ and
$(\Lambda_0, \Lambda_1)$ are equivalent in the sense that there is,
microlocally, a canonical transformation $\chi$ which maps
$(\Lambda_0, \Lambda_1)$ into $(\tilde{\Lambda}_0,
\tilde{\Lambda}_1)$ and $\chi(\Lambda_0 \cap \Lambda_1)=(\tilde{\Lambda}_0 \cap \tilde{\Lambda}_1)$.
This leads us to the following model case.

\begin{example}\label{example:basic} Let
$\tilde{\Lambda}_0=\Delta_{T^*\rr^n}=\{ (x, \xi; x, \xi)\st x \in \rr^n, \
\xi \in \rr^n \smo \}$ be the diagonal in $T^*\rr^n \times T^*\rr^n$
and let $\tilde{\Lambda}_1= \{ (x', x_n, \xi', 0; x', y_n, \xi', 0)\st
x' \in \rr^{n-1}, \ \xi' \in \rr^{n-1} \smo \}$.  Then,
$\tilde{\Lambda}_0$ intersects $ \tilde{\Lambda}_1$ cleanly in
codimension $1$.
\end{example}

Now we define the class of product-type symbols $S^{p,l}(m,n,k)$.

\begin{definition}
$S^{p,l}(m,n,k)$ is the set of all functions $a(z;\xi,\sigma) \in
C^{\infty} (\rr^m \times (\rr^n \smo) \times \rr^k )$ such that for every $K
\subset \rr^m$ and every $\alpha \in \zz^m_+, \beta \in \zz^n_+, \delta \in
\zz^k_+$ there is $c_{K, \alpha, \beta}$ such that
\[|\partial_z^{\alpha}\partial_{\xi}^{\beta}\partial_{\sigma}^{\delta}
a(z,\xi,\sigma)| \le c_{K,\alpha,\beta}(1+ |\xi|)^{p- |\beta|} (1+
|\sigma|)^{l-| \delta|}\] for all $(z,\xi,\tau) \in K \times
(\rr^n \smo) \times \rr^k$.
\end{definition}

Since any two sets of cleanly intersecting Lagrangians are equivalent,
we first define $I^{p,l}$ classes for the case in Example
\ref{example:basic}.

\begin{definition}[\cite{Guillemin-Uhlmann}]\label{def:GU}   Let $I^{p,l}(\tilde{\Lambda}_0,
\tilde{\Lambda}_1)$ be the set of all distributions $u$ such that
$u=u_1 + u_2$ with $u_1 \in C^{\infty}_0$ and $$u_2(x,y)=\int
e^{i((x'-y')\cdot \xi'+(x_n-y_n-s) \cdot \xi_n+ s \cdot \sigma)}
a(x,y,s; \xi,\sigma)d\xi d\sigma ds$$ with $a \in S^{p',l'}$ where
$p'=p-\frac{n}{2}+\frac{1}{2}$ and $l'=l-\frac{1}{2}$.
\end{definition}

This allows us to define the $I^{p,l}(\Lambda_0, \Lambda_1)$ class for
any two cleanly intersecting Lagrangians in codimension $1$ using the
microlocal equivalence with the case in Example \ref{example:basic}.

\begin{definition}\label{def:Ipl}\cite{Guillemin-Uhlmann} Let  $I^{p,l}( \Lambda_0, \Lambda_1)$
be the set of all distributions $u$ such that $ u=u_1 + u_2 + \sum
v_i$ where $u_1 \in I^{p+l}(\Lambda_0 \setminus \Lambda_1)$, $u_2 \in
I^{p}(\Lambda_1 \setminus \Lambda_0)$, the sum $\sum v_i$ is locally
finite and $v_i=Aw_i$ where $A$ is a zero order FIO associated to
$\chi ^{-1}$, the canonical transformation from above, and $w_i \in
I^{p,l}(\tilde {\Lambda}_0, \tilde{\Lambda}_1)$.
\end{definition}

This class of distributions is invariant under FIOs associated to
canonical transformations which map the pair $(\Lambda_0, \Lambda_1)$
to itself, whilst also preserving the intersection.  By definition, $F\in I^{p,l}(\Lambda_{0},\Lambda_{1})$ if
its Schwartz kernel belongs to $I^{p,l}(\Lambda_{0},\Lambda_{1})$. If
$F \in I^{p,l}(\Lambda_0, \Lambda_1)$ then $F \in I^{p+l}(\Lambda_0
\setminus \Lambda_1)$ and $F \in I^p(\Lambda_1 \setminus \Lambda_0)$
\cite{Guillemin-Uhlmann}.  Here by $F\in I^{p+l}(\Lambda_{0} \setminus
\Lambda_{1})$, we mean that the Schwartz kernel of $F$ belongs to
$I^{p+l}(\Lambda_{0})$ microlocally away from $\Lambda_{1}$.

To show that a distribution belongs to $I^{p,l}$ class we
use the iterated regularity property:

\begin{theorem} [{\cite[Proposition 1.35]{GU1990a}}]\label{thm:iterated
regularity}
If $u \in \Dc'(X \times Y)$ then $u \in I^{p,l} (\Lambda_0,
\Lambda_1)$ if there is an $s_0 \in \rr$ such that for all first order
pseudodifferential operators $P_i$ with principal symbols vanishing on
$ \Lambda_0 \cup \Lambda_1$, we have $P_1P_2 \dots P_r u \in
H^{s_0}_{loc}$.
\end{theorem}

In section \ref{sect:alpha-}, we will use the following theorem.
\begin{theorem}[{\cite{RF1, Nolan-fold_caustics}}] \label{thm:Felea Nolan}If
$F$ is a FIO of order m whose canonical relation is a two-sided fold
then $F^*F \in I^{2m,0}(\Delta, \tilde{C})$ where $\tilde{C}$ is
another two-sided fold.
\end{theorem}

\section{Summary of the main result for the case $\A=-1$}\label{sect:common_midpoint}

Recall that in the statement of Theorem \ref{thm:G*G}, we assumed that $\A\neq -1$. In fact, as already mentioned in the introduction, the case when $\A=-1$ in the context of  Theorem \ref{thm:G*G} was analyzed in our earlier paper \cite{AFKNQ:common_midpoint}, and the results obtained in this work can be considered as a bifurcation of the singularities that appear for the case when $-1\neq \A<0$, as $\A\to -1$. With this is mind, we state the main result obtained in \cite{AFKNQ:common_midpoint}, and briefly explain how our earlier result fits into the framework of the current article.

Let $\Tc$ denote the operator
\Beq\label{def:Gcm} \Tc V(s,t):=\int
e^{-i\omega\left(t-\norm{x-\gt(s)}-\norm{x-\gr(s)}\right)}
a(s,t,x,\omega) V(x) \D x \D \omega, \Eeq
where
\begin{align*}
\g_{T}(s)=\sqrt{(x_{1}+s)^2+x_{2}^2+h^2} \mbox{ and } \g_{R}(s)=\sqrt{(x_{1}-s)^2+x_{2}^2+h^2}.
\end{align*}

\begin{theorem}\cite{AFKNQ:common_midpoint}
Let $\Tc$ be the operator in \eqref{def:Gcm}. Then the normal operator $\Tc^{*}\Tc$ can be decomposed as a sum:

\[
\Tc^{*}\Tc\in I^{2m,0}(\Delta,\Lambda_1)+I^{2m,0}(\Delta, \Lambda_2)+
I^{2m,0}(\Lambda_{1},\Lambda_{3})+I^{2m,0}(\Lambda_{2},\Lambda_{3}).
\]
\end{theorem}
Here $\Lambda_{1}, \Lambda_{2}$ and $\Lambda_{3}$ denote the additional singularities (artifacts) caused due to reflection about the $x_{1}$ axis, reflection about the $x_{2}$ axis and rotation by $\pi$ about the origin, respectively. In other words,
$\Lambda_{1}$ is the same as the one defined in \eqref{def:C1} and $\Lambda_{2}$ and $\Lambda_{3}$ are defined as:
\[
\Lambda_2=\{(x_1, x_2, \xi_1, \xi_2; -x_1, x_2, -\xi_1,
\xi_2)\st (x,\xi)\in T^{*}X \smo \}
\]
and
\[
\Lambda_3=\{(x_1, x_2, \xi_1, \xi_2; -x_1, -x_2, -\xi_1,
-\xi_2)\st (x,\xi)\in T^{*}X \smo \}.
\]

The theorem above is a limiting case of the result in Theorem \ref{thm:G*G}, and in the limit, there is a bifurcation of the singularities, due to the presence of $4$ Lagrangians in the result above compared to $3$ in Theorem \ref{thm:G*G}. From \eqref{def:Sigma2}, we see that when $\A=-1$, the circle given by the equation
\[
\frac{\A}{A^{2}}+\frac{1}{B^{2}}=0
\]
becomes the straight line $x_{1}=0$ regardless of the value $s$, and as noted in Remark 2.3 of \cite{AFKNQ:common_midpoint}, the canonical relation associated to the operator $\Tc$ is a 4-1 relation due to the symmetries with respect to both the $x_{1}$ and $x_{2}$ axes. When $\A\neq -1$, the additional symmetry (about the $x_{2}$ axis) in the canonical relation of $\Tc$ is broken. This was one of our main motivations for the results in this paper.

\section{Analysis of the  operator $\Fc$ and the imaging operator
$\Fc^{*}\Fc$ for $\A\geq 0$}
\label{sect:F}

In this section, we prove  Theorems \ref{thm:F alpha>=0}  and
\ref{thm:F*F alpha>=0}.

 \begin{proof}[Proof of Theorem \ref{thm:F alpha>=0}] We first prove
that
\[\varphi:= -\omega\left(t-\sqrt{(x_1-\alpha
s)^2+x_2^2+h^2}-\sqrt{(x_1-s)^2+x_2^2+h^2}\right)\] is a
non-degenerate phase function.  We have that $\vp$ is a
phase function in the sense of H\"{o}rmander \cite{Ho1971}
because $\nabla_{x}\varphi\neq 0$ at points where the amplitude of the
operator $\Fc$ is elliptic. The differential $\n_{x}\vp$ vanishes at a
point on the ground directly ``between'' the source and receiver and
this point is given by $ \paren{\frac{(\A+1)s}{2},0}$.  However, in a
neighborhood of such points the amplitude $a$ vanishes due to the
cutoff function $f$ in the definition of $\Fc$ given in \eqref{def:F}.
Also we have that $\nabla_{s,t}\varphi$ is nowhere $0$ since $\omega
\neq 0$. The same reason that $\n_{x}\vp$ is non-vanishing at points
where the amplitude $a$ is elliptic also gives that $\vp$ is
non-degenerate.  Since $a$ satisfies an amplitude estimate, $\Fc$ is
an FIO \cite{DuistermaatBook}.  Finally since $a$ is of order $2$, the
order of the FIO is $3/2$ \cite[Definition 3.2.2]{DuistermaatBook}.
By definition \cite[Equation (3.1.2)]{Ho1971}
\[\CF =\Big{\{}(s,t,\partial_{s,t }\varphi(x,s,t,\omega)); \ (x,-\partial_{x}\varphi(x,s,t)); \
\partial_{\omega}\varphi(x,s,t,\omega)=0\Big{\}}.\] This
establishes \eqref{def:Canonical relation}. Furthermore, it is easy to see
that $(x_{1},x_{2},s,\omega)$ is a global parametrization of $\CF$.

Now we prove the claims about the canonical left and right projections from  $\CF$, the final parts of
Theorem \ref{thm:F alpha>=0}.
In the parameterization of $\CF$, we have
\[
\pi_L(x_1, x_2, s, \omega)=\left(s,A+B,-\left(\frac{x_1-\alpha s}{A} \alpha
+\frac{x_1-s}{B} \right)\omega, -\omega\right)
\]
 and the derivative is
{\small{\[{  \lb d \pi_L\rb}= \left(\begin{matrix}
0 & 0 & 1 & 0\\
\frac{x_1- \alpha s}{A} + \frac{x_1-s}{B}& \frac{x_2}{A} +
\frac{x_2}{B} & \ast & 0 \\
-\omega\left(\frac{x_2^2+h^2}{A^3}\alpha +\frac{x_2^2+h^2}{B^3}\right) &
\omega\left(\frac{\alpha(x_1- \alpha s)x_2}{A^3} +
\frac{(x_1-s)x_2}{B^3}\right) & \ast
& \ast\\
0 & 0 & 0 & -1
\end{matrix} \right).
\]}}
Then
\Beq\label{Determinant-left-projection}
\det \lb d \pi_L\rb=-\omega x_2 \left(\frac{\alpha}{A^2} +
\frac{1}{B^2}\right) \left(1+ \frac{(\gamma_T-x) \cdot
(\gamma_R-x)}{AB}\right).
\Eeq

The third term would be zero when the unit vectors $(\gt(s)-x)/A$ and
$(\gr(s)-x)/B$ point in opposite directions, but this cannot happen
since the transmitter and receiver are above the plane of the Earth.
Also since $\A>0$, $\left(\frac{\alpha}{A^2} +
\frac{1}{B^2}\right)\neq 0$. Hence, this determinant vanishes to first
order when $x_2=0$.  This corresponds to $\Sigma_1$ given in
\eqref{def:Sigma1}.

On $\Sigma_1$ the kernel of $\lb d \pi_L\rb $ is spanned by $\frac{\partial}{\partial
x_2}  \not \subset T\Sigma_1$. So $\pi_L$ has a fold
singularity along $\Sigma_1$.

Similarly, we have,
\[
\pi_R(x_1, x_2, s, \omega)=\left(x_1,x_2,-\left(\frac{x_1- \alpha
s}{A}+\frac{x_1-s}{B} \right)\omega,-\left(\frac{x_2}{A}+\frac{x_2}{B}\right) \omega\right).
\]
Then
\[
{\lb d \pi_R\rb}= \left(\begin{matrix} 1 & 0 & 0 & 0\\
0 & 1 & 0 & 0\\
\ast & \ast & \omega(\frac{x_2^2+h^2}{A^3}\alpha +
\frac{x_2^2+h^2}{B^3})& -(\frac{x_1-\alpha s}{A} + \frac{x_1-s}{B})\\
\ast & \ast & -\omega(\frac{(x_1-\alpha s)x_2}{A^3}\alpha +
\frac{(x_1-s)x_2}{B^3}) & -(\frac{x_2}{A} + \frac{x_2}{B})
\end{matrix} \right)
\]
has the same determinant  as $\lb d \pi_{L}\rb$.Therefore $\pi_R$ drops rank simply by one on
$\Sigma_{1}$. On $\Sigma_1$, the kernel of $\pi_R$ is spanned by
$\frac{\partial}{\partial \omega}$ and $\frac{\partial}{\partial
s}$ which are tangent to $\Sigma_1 $. Thus $\pi_R$ has a blowdown
singularity along $\Sigma_1$.

\end{proof}

Next we analyze the imaging operator $\Fc^{*}\Fc$. We have the
following integral representation for $\Fc^{*}\Fc$:
\[\Fc^{*}\Fc V(x)=\int  e^{\I
\tilde{\phi}\paren{x,s,t,\omega,\tilde{\omega},y}}
 \overline{a(s,t,x,\omega)}a(s,t,y,\tilde{\omega})V(y)\D s \D t \D
\omega \D  \tilde{\omega} \D y, \]
where
\begin{align*}\tilde{\phi}
 =& \big{(}
\omega\left(t-(\norm{x-\gamma_{T}(s)}+\norm{x-\gamma_{R}(s)})\right)\\
 &\quad \ \,-\tilde{\omega}\left(t-\left(\norm{y-\gamma_{T}(s)}+\norm{y-\gamma_{R}(s)}\right)
 \right)
\big{)}
\end{align*}
After an application of the method of stationary phase in $t$ and
$\tilde{\omega}$, the Schwartz kernel of this operator becomes
\Beq\label{Kernel of composed operator} K(x,y)=\int
e^{i\Phi(y,s,x,\om)}
\tilde{a}(y,s,x,\omega)\:
\D s \D \omega.
\Eeq
 where \bel{Phase of
 kernel}\begin{aligned} \Phi(y,s,x,\om)=&\omega\big{(}
 \norm{y-\gamma_{T}(s)}+\norm{y-\gamma_{R}(s)}
\\&\qquad-(\norm{x-\gamma_{T}(s)}+\norm{x-\gamma_{R}(s)})\big{)}.
\end{aligned}\ee
 Note that $\tilde{a}\in S^{4}$ since we have assumed that
 $a\in S^{2}$.

 \begin{proposition}\label{Wavefront of composition I}  Let $\A\geq 0$.
The wavefront relation of  the kernel $K$ of $\Fc^{*}\Fc$ satisfies
\[
WF'(K)\subset \Delta \cup C_1,
\]
where $\Delta$ is the diagonal in $T^*X \times T^*X$, and $C_1$ is
given by \eqref{def:C1}. We have that $\Delta$ and $C_1$ intersect cleanly in
codimension 2 in $\Delta$ or $C_1$.  \end{proposition}

\bpr
Let $(s, t, \sigma, \tau;
y, \eta) \in \CF$. Then we have \bel{Canonical relation of G}\begin{aligned}
&t=\sqrt{(y_1-\alpha s)^2+y_2^2+h^2}+ \sqrt{(y_1-s)^2+y_2^2+h^2} \\
&\sigma= \tau\left(\frac{y_1-\alpha s}{\sqrt{(y_1- \alpha
s)^2+y_2^2+h^2}} \alpha
+\frac{y_1-s}{\sqrt{(y_1-s)^2+y_2^2+h^2}}\right)
\\
&\eta_1=\tau \left(\frac{y_1- \alpha s}{\sqrt{(y_1- \alpha
s)^2+y_2^2+h^2}}+\frac{y_1-s}{\sqrt{(y_1-s)^2+y_2^2+h^2}}\right)
\\
&\eta_2=\tau \left(\frac{y_2}{\sqrt{(y_1- \alpha
s)^2+y_2^2+h^2}}+\frac{y_2}{\sqrt{(y_1-s)^2+y_2^2+h^2}}\right)
\end{aligned}\end{equation}

and $(x, \xi; s, t, \sigma, \tau) \in C^t$ implies
\begin{align}\label{Canonical relation of G*}
\notag&t=\sqrt{(x_1-\alpha s)^2+x_2^2+h^2}+
\sqrt{(x_1-s)^2+x_2^2+h^2}\\
\notag&\sigma= \tau\left(\frac{x_1-\alpha s}{\sqrt{(x_1- \alpha
s)^2+x_2^2+h^2}}\alpha +\frac{x_1-s}{\sqrt{(x_1-s)^2+x_2^2+h^2}}\right)\\
&\xi_1=\tau \left(\frac{x_1-\alpha s}{\sqrt{(x_1-\alpha
s)^2+x_2^2+h^2}}+\frac{x_1-s}{\sqrt{(x_1-s)^2+x_2^2+h^2}}\right)\\
\notag&\xi_2=\tau\left (\frac{x_2}{\sqrt{(x_1- \alpha
s)^2+x_2^2+h^2}}+\frac{x_2}{\sqrt{(x_1-s)^2+x_2^2+h^2}}\right)
\end{align}

From the first two relations in \eqref{Canonical relation of G} and
\eqref{Canonical relation of G*}, we have
\bel{Isorange curves}
\begin{gathered}
\sqrt{(y_1-\alpha s)^2+y_2^2+h^2}+
\sqrt{(y_1-s)^2+y_2^2+h^2}\qquad
\\
\qquad=\sqrt{(x_1- \alpha s)^2+x_2^2+h^2}+
\sqrt{(x_1-s)^2+x_2^2+h^2}
\end{gathered}\end{equation}

and \bel{Isodoppler curves} \begin{gathered}\frac{y_1-\alpha s}{\sqrt{(y_1-\alpha
s)^2+y_2^2+h^2}}\alpha
+\frac{y_1-s}{\sqrt{(y_1-s)^2+y_2^2+h^2}}\qquad\\
\qquad=
\frac{x_1- \alpha s}{\sqrt{(x_1-\alpha s)^2+x_2^2+h^2}}\alpha
+\frac{x_1-s}{\sqrt{(x_1-s)^2+x_2^2+h^2}}.\end{gathered}\ee

 We will use prolate spheroidal coordinates with foci $\gr(s)$ and
$\gt(s)$ to solve for $x$ and $y$. We let
\begin{align}\label{Prolate coordinate system}
\begin{array}{ll}
x_{1}= \frac{1+\alpha}{2}s + \frac{1-\alpha}{2}s\cosh \rho \cos \phi &
y_{1}=\frac{1+\alpha}{2}s + \frac{1-\alpha}{2}s\cosh \rho' \cos \phi'\\
x_{2}=\frac{1-\alpha}{2}s\sinh \rho \sin \phi \cos \theta & y_{2}=
\frac{1-\alpha}{2}s\sinh \rho' \sin \phi' \cos \theta'\\
x_{3}=h+ \frac{1-\alpha}{2}s\sinh \rho \sin \phi \sin \theta &
y_{3}=h+ \frac{1-\alpha}{2}s\sinh \rho' \sin \phi' \sin \theta'
\end{array}
\end{align}
with $\rho$ and $\rho'$ positive, $\phi$ and $\phi'$ in the interval
$[0,\pi]$ and $\theta$ and $\theta'$ in $[0,2\pi]$.  In this case
$x_3=0$ and we use it to solve for $h$. Hence
$$A^2=(x_1-\alpha s)^2+x_2^2+h^2=\frac{(1-\alpha)^2}{4}s^2(\cosh \rho+
\cos \phi)^2$$ and
$$B^2=(x_1-s)^2+x_2^2+h^2=\frac{(1-\alpha)^2}{4}s^2(\cosh \rho- \cos
\phi)^2.$$
Noting that $s>0$ and $\cosh\rho\pm\cos\phi>0$, the first
relation given by \eqref{Isorange curves} in these coordinates becomes
$$s(\cosh \rho- \cos \phi)+s(\cosh \rho+ \cos \phi)=s(\cosh \rho'- \cos \phi')+s(\cosh \rho'+ \cos \phi') $$ from which we get $$\cosh \rho=\cosh \rho' \Rightarrow \rho=\rho'.$$

The second relation, given by \eqref{Isodoppler curves}, becomes
{\small{\[\frac{\cosh \rho \cos \phi -1}{\cosh \rho-\cos \phi}
+\alpha \frac{\cosh \rho \cos \phi +1}{\cosh \rho+\cos
\phi}=\frac{\cosh \rho \cos \phi' -1}{\cosh \rho-\cos \phi'} + \alpha
\frac{\cosh \rho \cos \phi' +1}{\cosh \rho+\cos \phi'} \]}}
After simplification we get \bel{composition-in-coords}\begin{aligned}(\cos \phi -\cos
\phi')&[(\alpha+1)(\cosh^2 \rho+\cos \phi \cos
\phi')\\
&\qquad-(\alpha-1)\cosh\rho(\cos \phi+ \cos \phi')] =0\end{aligned}\ee
which implies either that \bel{C1} \cos \phi =\cos \phi' \mbox{ which
implies } \phi =\phi'\ee (since we can assume $\phi\in [\pi,2\pi]$ for
points on the ground) or that \bel{C2} (\alpha+1)(\cosh^2 \rho+\cos
\phi \cos \phi')-(\alpha-1)\cosh\rho(\cos \phi+ \cos
\phi')=0=\frac{\alpha}{A \tilde{A}}+\frac{1}{B\tilde{B}} \ee where
$\tilde{A}$ and $\tilde{B}$ are defined as in \eqref{def:AB} but
evaluated at $(s,y)$ and the third term in the equality is equivalent
to the first term.

We consider Conditions \eqref{C1} and \eqref{C2} separately.  First,
assume Condition \eqref{C1} holds. Then we have $\phi=\phi'$.  In this
case, \[\cos \theta= \pm \sqrt{1-\frac{h^2}{s^2\sinh^2 \rho \sin^2
\phi}} =\pm \cos \theta'\] and note that $x_3=0$ implies that $\sin
\phi\neq 0$. We also remark that it is enough to consider $\cos
\theta=\cos \theta'$ as no additional relations are introduced by
considering $\cos\theta=-\cos\theta'$ over the relations we now
address.  Now we go back to $x$ and $y$ coordinates.  If
$\theta=\theta'$ then $x_1=y_1, \ x_2=y_2, \xi_i=\eta_i$ for $i=1,2$.
In this case, the composition, $\CF^t \circ \CF \subset \Delta=\{ (x,
\xi; x, \xi) \} $.  If $\theta'=\pi-\theta$ then $x_1=y_1, \ -x_2=y_2,
\xi_1=\eta_1, -\xi_2=\eta_2$. For these points the composition, $\CF^t
\circ \CF$ is a subset of $C_1$ in \eqref{def:C1}. Note that
\eqref{C2} has no solutions for $\alpha\geq 0$. The statements about
clean intersection are the same as the one given in
\cite{AFKNQ:common_midpoint}. This concludes the proof of Proposition
\ref{Wavefront of composition I}.

\epr

\begin{proof}[Proof of Theorem \ref{thm:F*F alpha>=0}]
  We will use the iterated regularity method (Theorem
  \ref{thm:iterated regularity}) to show that the kernel of
  $\Fc^*\Fc \in I^{3,0}(\Delta, C_1)$.  We  consider the generators of the
  ideal of functions that vanish on $\Delta \cup C_1$ \cite{RF1}. These are given by
\bel{generators:F}\begin{aligned}
&\tilde{p}_{1}=x_{1}-y_{1},\quad
\tilde{p}_{2}=x_{2}^{2}-y_{2}^{2},\quad
\tilde{p}_{3}=\xi_{1}-\eta_{1},\\
&\tilde{p}_{4}=(x_{2}-y_{2})(\xi_{2}+\eta_{2}),\quad\tilde{p}_{5}=(x_{2}+y_{2})(\xi_{2}-\eta_{2}),\\
&\tilde{p}_{6}=\xi_{2}^{2}-\eta_{2}^{2}.
\end{aligned}\ee
We show in Appendix \ref{proofs:iterated regularity} that each
$\tilde{p}_{i}$ can be expressed as sums of products of
$\partial_\omega{\Phi}$ and $\partial_s{\Phi}$ with smooth functions.
Let $p_{i}=q_{i}\tilde{p}_{i}$, for $1\leq i\leq 6$, where $q_{1},
q_{2}$ are homogeneous of degree $1$ in $(\xi,\eta)$, $q_{3},q_{4}$
and $q_{5}$ are homogeneous of degree $0$ in $(\xi,\eta)$ and $q_{6}$
is homogeneous of degree $-1$ in $(\xi,\eta)$. Let $P_{i}$ be
pseudodifferential operators with principal symbols $p_{i}$ for $1\leq
i\leq 6$. The $\tilde{p}_i$ and arguments using iterated
regularity are similar to those used in \cite[Thm.\ 1.6]{RF1}  and
in \cite[Thm.\ 4.3]{Krishnan-Quinto}.

 We use the same arguments as in \cite{AFKNQ:common_midpoint}  to show that the orders $p,l$ from $I^{p,l}(\Delta, C_1)$ are $p=3$ and $l=0$.
\end{proof}

\section{Analysis of the forward operator $\Gc$ and the imaging operator
$\Gc^{*}\Gc$ for $\A< 0$}
\label{sect:alpha-}

In this section, we analyze the operator $\Gc$ in \eqref{def:G}. In
\cite{AFKNQ:common_midpoint} we
analyzed  the case $\alpha=-1$.  For the case with $\alpha<0$, we  make
another simplification:
\[\text{ Assume \ \ } \alpha<-1.\]
If $\alpha\in (-1,0)$, then we can reduce it to the case $\alpha<-1$
by using the diffeomorphisms $(x_1,x_2)\mapsto (-x_1,x_2)$ and
$(s,t)\mapsto (s/|\alpha|,t)$.

We first prove Theorem \ref{thm:G alpha<0}.
\bpr[Proof of Theorem \ref{thm:G alpha<0}] In the proof of this theorem, most of the statements are already proved in Theorem \ref{thm:F alpha>=0}. We just prove the statements regarding the properties of the projection maps $\pi_{L}$ and $\pi_{R}$. Recall from the proof of Theorem \ref{thm:F alpha>=0} that
\[{\D \pi_L}= \left(\begin{matrix}
0 & 0 & 1 & 0\\
\frac{x_1- \alpha s}{A} + \frac{x_1-s}{B}& \frac{x_2}{A} +
\frac{x_2}{B} & \ast & 0 \\
-\omega\left(\frac{x_2^2+h^2}{A^3}\alpha +\frac{x_2^2+h^2}{B^3}\right) &
\omega\left(\frac{\alpha(x_1- \alpha s)x_2}{A^3} +
\frac{(x_1-s)x_2}{B^3}\right) & \ast
& \ast\\
0 & 0 & 0 & 1
\end{matrix} \right)
\]
and $$\det \D \pi_L=\omega x_2 \left(\frac{\alpha}{A^2} +
\frac{1}{B^2}\right) \left(1+ \frac{(\gamma_T-x) \cdot
(\gamma_R-x)}{AB}\right)$$

Clearly this determinant drops rank
when the first term, $x_2=0$.  This corresponds to $\Sigma_1$ given by
\eqref{def:Sigma1}.

The determinant also drops rank when the
second term is zero, which can occur when $\alpha<0$; this corresponds
to $\Sigma_2$ given by \eqref{def:Sigma2}. Note that $\pi_L$ drops rank by 2 at the
intersection points of $\Sigma_1$ and $\Sigma_2$ (where $x_2=0$) but we exclude them using the cutoff function $g$ described preceding  \eqref{def:tpm}.

On $\Sigma_2$, using the first, second,
and fourth row of $d\pi_L$, the kernel is $\frac{\partial}{\partial
x_1}-\frac{\frac{x_1-\alpha s}{A}+\frac{x_1-s}{B}}{x_2(\frac{1}{A} +
\frac{1}{B})} \frac{\partial}{\partial x_2}$ which applied to
$\Sigma_2$ gives us $\frac{2 s (\alpha-1)}{(\alpha
+1)(\frac{1}{A}+ \frac{1}{B})} (\frac{\alpha}{A} -\frac{1}{B})$. We
have that $s \neq 0$ and $\alpha + 1 \neq 0$. If $\frac{\alpha}{A}
-\frac{1}{B}=0$ then using $\frac{\alpha}{A^2} +\frac{1}{B^2}=0$ we
get $\frac{\alpha (\alpha+1)}{A^2}=0$ which is a contradiction. Thus
 $\ker\paren{\D  \pi_L}  \not \subset T\Sigma_2$ which implies that $\pi_L$
has a fold singularity along $\Sigma_2$.  \vskip .5 cm \par Similarly,

\[
{\D \pi_R}= \left(\begin{matrix} 1 & 0 & 0 & 0\\
0 & 1 & 0 & 0\\
\ast & \ast & -\omega(\frac{x_2^2+h^2}{A^3}\alpha +
\frac{x_2^2+h^2}{B^3})& (\frac{x_1-\alpha s}{A} + \frac{x_1-s}{B})\\
\ast & \ast & \omega(\frac{(x_1-\alpha s)x_2}{A^3}\alpha +
\frac{(x_1-s)x_2}{B^3}) & (\frac{x_2}{A} + \frac{x_2}{B})
\end{matrix} \right)
\]
has the same determinant up to sign and so $\pi_R$ drops rank by one on
$\Sigma$.
On $\Sigma_2$, using the last row of
$d\pi_R$, the kernel is $ \frac{\partial}{\partial s}-
\omega\frac{\frac{x_1-\alpha s}{A^3} \alpha
+\frac{x_1-s}{B^3}}{\frac{1}{A} + \frac{1}{B}}
\frac{\partial}{\partial \omega}$ which applied to $\Sigma_2$ gives $2
\alpha(s-\frac{2x_1}{\alpha+1})$. If $s=\frac{2 x_1}{\alpha +1}$ then
from $\frac{\alpha}{A^2} +\frac{1}{B^2}=0$ we obtain
$s^2(\frac{\alpha-1}{2})^2+x_2^2+h^2=0$ which is a contradiction.
Hence  $\ker\paren{\D\pi_R} \not \subset T\Sigma_2$ which implies that
$\pi_R$ has a fold singularity along $\Sigma_2$ as well. This completes the proof of Theorem \ref{thm:G alpha<0}.
\epr

\begin{proposition}\label{Wavefront  of composition II} For $\alpha<0$,
the wavefront relation of the kernel $K$ of $\Gc^{*}\Gc$ satisfies,
\[
WF'(K)\subset \Delta \cup C_{1}\cup C_2,
\]
where $\Delta$ is the diagonal in $T^*X \times T^*X$, $C_{1}$ is given
by \eqref{def:C1} and $C_{2}$ is defined as
\bel{def:C2} \begin{aligned}C_2 =\Big\{(x,\xi;y,\xi')&\st
\exists
(s,t),\  (x,\xi)\in N^*(E(s,t)),\,
(y,\xi')\in N^*(E(s,t)),\\
&\frac{\alpha}{A\At}+\frac{1}{B \Bt}=0, (x_2,y_2)\neq (0,0)\Big\},\end{aligned} \ee where
$A=A(s,x)$ and $\At = A(s,y)$, $B=B(s,x)$ and $\Bt=B(s,y)$ and
$E(s,t)$ is given by \eqref{def:E}. Furthermore,
 $\Delta$ and $C_{1}$ intersect
cleanly in codimension 2, $\Delta$ and $C_2$ intersect
cleanly in codimension 1, $C_2$ and $C_{1}$ intersect
cleanly in codimension 1, and  $ \Delta \cap C_1 \cap C_2=
\emptyset$. \end{proposition}

 \bpr In fact, this proposition is already
proved in Proposition \ref{Wavefront of composition I}. Here, unlike
the situation in Proposition \ref{Wavefront of composition I}, there
is a nontrivial contribution to the wavefront of the composition from
\eqref{C2}. Hence for $\A<0$, we have that
\[
WF'(K)\subset \Delta \cup C_{1}\cup C_2,
\]
To show that no point in $C_2$ has $x_2=0=y_2$, one uses
\eqref{C2} and that $x_3=0=y_3$ in \eqref{Prolate coordinate system}.
Finally, note that $\Delta \cap C_1 \cap C_2=\emptyset$ since we
exclude the points of intersection of $\Sigma_1$ and $\Sigma_2$ due to
the cutoff function $g$ defined in Section \ref{sect:results alpha-}.
One can show that $C_2$ is an immersed conic Lagrangian manifold
that is a two-sided fold using Definition \ref{def:fold} and the proof
of Theorem \ref{ImagingSection:I(p,l) regularity} part
\eqref{G2*G2}).

Using Def.\ \ref{def:cleanly} and the calculations above, one can also
show that these manifolds intersect in the following ways:
\begin{enumerate}[(a)]
\item\label{Delta-C1} $\Delta$ intersects $C_1$ cleanly in codimension
2, \[\Delta \cap C_1=\left\{(x, \xi; x, \xi)\in \Delta \st
x_{2}=0=\xi_{2} \right\} .\] This is part of Proposition 5.1 in
\cite{AFKNQ:common_midpoint}.

\item $\Delta$ intersects $C_2$ cleanly  in
codimension 1, \[ \begin{aligned}\Delta \cap C_2&=\left\{(x, \xi; x, \xi)\in \Delta \st
\frac{\alpha}{A^2}+\frac{1}{B^2}=0 \right\}
\\&=\{(x,\xi;y,\eta)\in C_2\st
x_1=y_1,\, x_2=y_2\}.\end{aligned} \]  Note that the condition that $x_1=y_1$ in
$C_2$ implies $x_2 = \pm y_2$ and so the condition $x_2=y_2$ does not
increase the codimension of the intersection.  Using
 \cite{Melrose-Taylor}, one can show the intersection is clean.

 \item $C_1$ intersects $C_2$ cleanly in codimension 1, \[ C_1 \cap
C_2=\left\{(x, \xi; y, \eta) \in C_1\st \frac{\alpha}{A
\tilde{A}}+\frac{1}{B\tilde{B}}=0\right\}.\]  Using
 \cite{Melrose-Taylor}, one can show the intersection is clean.

\item $\Delta \cap C_1 \cap C_2=\emptyset$ since
we exclude the points of intersection of $\Sigma_1$ and $\Sigma_2$.
 \end{enumerate}

This completes the proof of the proposition.
\epr

For the rest of the proof, we focus on $C_2$.  Let $\beta=\sqrt{-\A}$,
then $\beta>1$.  Let $(x,\xi,y,\xi')\in C_2$.  then, by \eqref{def:C2}
there is an $(s,t)$ such that $x$ and $y$ are both in $E(s,t)$ and
\[\frac{\B B}{A}= \frac{\At}{\B \Bt}\] where $A,\At,B,\Bt$ are
given below \eqref{def:C2}. Therefore, if $(x,\xi,y,\xi')\in C_2$
then \bel{important}\begin{aligned}\exists (s,t)\in (0,\infty)^2,\
\exists k\in (0,\infty), \quad x, y\in E(s,t) \quad \frac{\B
B(s,x)}{A(s,x)}=k,\quad \frac{\B
B(s,y)}{A(s,y)}=\frac{1}{k}.\end{aligned}\ee A calculation shows that
if $k\neq \B$ then \bel{circle condition}\begin{gathered} \frac{\B
B(s,x)}{A(s,x)}=k\ \ \Leftrightarrow \ \ \lb x_1 - \frac{\B^{2}s
(1+k^{2})}{\B^{2}-k^{2}}\rb^{2}+x_2^{2}=\frac{\B^{2}
s^{2}k^{2}(\B^{2}+1)^{2}}{(\B^{2}-k^{2})^{2}}-h^{2} \end{gathered}\ee
If $k=\B$, then $\B B/A=k$ is the equation of a vertical line with
$x_1$ intercept $(1-\B^2)s/2$.  \medskip

We first use this characterization of $C_2$ to prove Statements
\eqref{O1} and \eqref{O3} of Theorem \ref{thm:G*G}. As already
mentioned, the diagonal relation $\Delta$ and $C_{1}$ given by
\eqref{def:C1} intersect cleanly in codimension 2 on either
submanifold. Hence there is a well-defined $I^{p,l}$ class associated
to $\Delta$ and $C_{1}$.

%NOTE TO SIAM: it seems hyperref doesn't work in section titles so I
%had to reword the title so I could use the \ref{} and \eqref{}
%outside of the ``\subsection

%please don't change \subsection{proof of Thm}\label.... it is needed
%when we submit it to SIAM

  \subsection{Proof of Theorem \textbf{\ref{thm:G*G},}\label{proof:O1}
Statement \eqref{O1}}\ \
Recall from statement \eqref{O1} of this
theorem that the function $r_1$ is a cutoff function compactly
supported in \bel{def:O1}O_1 = \sparen{(s,t)\st 0<s<s_0 \mbox{ and }
0<t<\infty}\ee where $s_0$ (see \eqref{def:s0-alpha}) can be written
in terms of $\B$ as \bel{def:s0} s_0:
=\frac{h(\B^2-1)}{\B(\B^2+1)}.\ee

We show that for $(s,t)\in O_1$, there are no $x$ and $y$ satisfying
\eqref{important}.  Therefore, $C_2\cap WF'(K_1)=\emptyset$
where $K_1$ is the Schwartz kernel of $\Gc^*r_1 \Gc$.

So, assume for some $(s,t)\in O_1$ \eqref{important} holds.  Then, the
right-hand side of \eqref{circle condition} can be estimated by
\[
\frac{\B^{2}
s^{2}k^{2}(\B^{2}+1)^{2}}{(\B^{2}-k^{2})^{2}}-h^{2}<h^{2}\lb \frac{(\B^2-1)^2k^2-(\B^2-k^2)^2}{(\B^2-k^2)^2}\rb= h^{2}\lb \frac{(k^2-1)(\B^4-k^2)}{(\B^2-k^2)^2}\rb.
\]
  Since $\B>1$, if $0<k\leq 1$, this calculation and \eqref{circle
condition} shows that $\frac{\B B}{A}=k$ has no solution.  Therefore
there are no solutions to \eqref{important} if $0<k\leq 1$. Now
assume for some $k>1$, $\frac{\B B(s,x)}{A(s,x)} = k$, then for
\eqref{important} to have a solution that means that there must be a
point $y\in E(s,t)$ with $\frac{\B B(s,y)}{A(s,y)} = 1/k$.
However, this is impossible since $0<1/k\leq 1$. This shows that for
$(s,t)\in O_1$, there is no solution to \eqref{important}.

Now following the
proof of Theorem \ref{thm:F*F alpha>=0}, we achieve the result of
Statement \eqref{O1}.
\qed

\subsection{Proof of Theorem \textbf{\ref{thm:G*G},}\label{proof:O3}
Statement \eqref{O3}}\ \
Recall that the cutoff function $r_3(s,t)$ is compactly supported in
\Beq\label{def:O3} O_3 = \Big{\{}(s,t): s_0<s<\infty \mbox{ and }
t<t_{s}^{-} \mbox{ or } t>t_{s}^{+}\Big{\}} \Eeq where $t_{s}^{\pm}$
is defined in \eqref{def:tpm}.  The operator we analyze in this part
of the proof is $\Gc^* r_3 \Gc$.

We define the following set \bel{def:C(s,k)} C(s,k):=\sparen{x\st
\frac{\B B(s,x)}{A(s,x)}=k}.  \ee Note that $C(s,1)= \Stwox(s)$,
$C(s,\B)$ is the vertical line $x_1 = (1+\A)s/2=(1-\B^2)s/2$, and if
$k$ is small enough, $C(s,k)=\emptyset$.  By \eqref{circle condition},
when $k\neq \B$ and $C(s,k)\neq \emptyset$ then $C(s,k)$ is the circle
centered at $\paren{\frac{\B^{2}s (1+k^{2})}{\B^{2}-k^{2}},0}$ and of
radius \bel{def:r} r(s,k):=\sqrt{\frac{\B^{2}
s^{2}k^{2}(\B^{2}+1)^{2}}{(\B^{2}-k^{2})^{2}}-h^{2}}.\ee

Let $(s,t)\in O_3$.  If there were a solution $(x,y)$ to
\eqref{important} for some $k$, then $x\in E(s,t)\cap C(s,k)$ (as $\B
B/A=k$ on $C(s,k)$)
\emph{and} $y\in E(s,t)\cap C(s,1/k)$.  If $t>t_s^+$ then the ellipse
$E(s,t)$ encloses $\Stwox(s)$ by a calculation.  Therefore, by the
final statement of Lemma \ref{lemma:circles}, $E(s,t)$ meets no circle
$C(s,k)$ for $k\in(0,1]$ and so there is no solution to
\eqref{important}.  Now, if $t<t_s^-$ then the ellipse $E(s,t)$ is
enclosed by $\Stwox(s)$ and, by the final statement of Lemma
\ref{lemma:circles}, $E(s,t)$ meets no $C(s,k)$ for $k\in(1,\infty)$
and so there is no solution to \eqref{important} in this case, too.
Therefore, $C_2\cap WF'(K_3)=\emptyset$ where $K_3$ is the
Schwartz kernel of $\Gc^*r_3 \Gc$. Now proceeding as in the proof of
Theorem \ref{thm:F*F alpha>=0}, we complete the proof of Statement
\eqref{O3} of Theorem \ref{thm:G*G}.  \qed

The rest of this section is devoted to the proof of Statement
\eqref{O2} of Theorem \ref{thm:G*G}.

\subsection{Proof of Theorem \textbf{\ref{thm:G*G}},
Statement \eqref{O2}}\label{proof:O2}\ \ The reconstruction operator we consider in
statement \eqref{O2} of Theorem \ref{thm:G*G} is $\Gc^* r_2 \Gc$ where
the mute $r_2$ has compact support in \bel{def:O2}O_2=\{(s,t):
s_0<s<\infty\ \text{ and }\ t_{s}^{-}<t<t_{s}^{+}\}\ee where
$t_{s}^{\pm}$ is defined in \eqref{def:tpm} and where $s_0$ is defined
by \eqref{def:s0}.

Recall that the canonical relation of $\Gc$ drops rank on the union of
two sets, $\Sigma_{1}$ and $\Sigma_{2}$. Accordingly, we decompose
$\Gc$ into components such that the canonical relation of each
component is either supported near a subset of the union of these two
sets, one of these two sets or away from both these sets. To do this,
we define several cutoff functions.

\subsubsection{The primary cutoff functions $\psi_1$ and
$\psi_2$}\label{sect:cutoffs}

The cutoff
 $\psi_1(x)$ will be equal to $1$ near the
$x_1$-axis and zero away from it, and  $\psi_2(s,x)$ will be
 equal to one near $\Stwox(s)$ and equal to zero away from it as in
 Figure \ref{fig:supp psis}.
\begin{figure}[h!]
\begin{center}\includegraphics[width=4in]{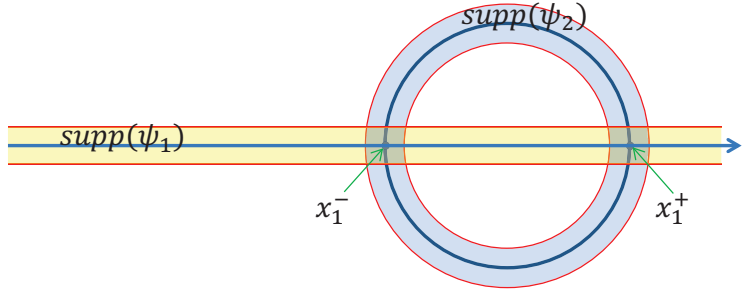}
\end{center}\caption{Picture of $\supp(\psi_1)$ and $\supp(\psi_2)$.
Note that the middle circle is $\Stwox(s)$ and the circles are
not exactly concentric.}\label{fig:supp psis}
\end{figure}

To define these functions precisely, we need to set up some
preliminary relations.  Because the mute function $r_2$ is zero near
$s_0$ and has compact support, there is an $s_1> s_0$ such that $r_2$
is zero for $s\leq s_1$ and all $t$.  Because the radius $r(s_1,1)>0$
and the function $r$ is continuous, there is a $k_1\in (0,1)$ such
that $r(s_1,k_1)>0$.  Since $r$ (see \eqref{def:r}) is an increasing
function in $s$ and $k$ separately, we can choose $\eps>0$ such that
\bel{eps condition} r(s,k)\geq r(s_1,k_1)>12\eps \ \ \text{for $k\geq
k_1$, and $s\geq s_1$.}\ee Without loss of generality, we can assume
\bel{eps estimate}\eps<\frac{\min(\B-1,1-k_1, 1/4)}{6}.\ee

Now, let
$\psi_{1}$ be an infinitely differentiable function defined as
follows: \Beq\label{def:psi1} \psi_1(x)= \begin{cases} 1, & |x_2| <
\epsilon \\
0, & |x_2| > 2 \epsilon
\end{cases}
\Eeq and we extend this function smoothly between $0$ and $1$.

For $s>\so$, let $k_0(s)$ be defined by \bel{def:k0} r(s,k_0(s))=0.\ee
Note that $k_0(s)$ can be explicitly calculated using \eqref{def:r}.
So, if $k>k_0(s)$, $C(s,k)$ is a nontrivial circle.  Finally, note
that if $s\geq s_1$, then $k_1>k_0(s)$; this is true because
$r(s,k_1)\geq r(s_1,k_1)>0$ for such $s$.

To define $\psi_2$ we first prove a lemma about the circles $C(s,k)$.

\begin{lemma}\label{lemma:circles}
Let $s\geq s_1$.

\begin{enumerate}
\item\label{to left}If $k>\beta$ then $C(s,k)$ is to the left of
the vertical line $C(s,\beta)$ which is to the left of $C(s,\ell)$ for
any $\ell\in (k_0(s),\B)$.

\item \label{contain} If $k_0(s)<j<k<\B$ then $C(s,j)$ is
contained inside $C(s,k)$, and these circles do not intersect.

\item \label{foliate} For any
$\delta\in (0,6\eps)$,
\[\sparen{x\st \abs{\frac{\B B(s,x)}{A(s,x)}}<\delta} = \bigcup_{k\in I} C(s,k)\] is an open set containing
 $\Stwox(s)=C(s,1)$.
\end{enumerate}
\end{lemma}

\begin{proof}
 Statement \eqref{to left} of the lemma is a straightforward
calculation.

Now, fix $s\geq s_1$.  Let $k\in (k_0(s),\B)$, then the endpoints of
$C(s,k)$ on the $x_1$-axis are
\[\begin{aligned} \xl(k) &= \frac{\B^{2}s
(1+k^{2})}{\B^{2}-k^{2}}-\sqrt{\frac{\B^{2}
s^{2}k^{2}(\B^{2}+1)^{2}}{(\B^{2}-k^{2})^{2}}-h^{2}}\\
\xr(k) &= \frac{\B^{2}s (1+k^{2})}{\B^{2}-k^{2}}+\sqrt{\frac{\B^{2}
s^{2}k^{2}(\B^{2}+1)^{2}}{(\B^{2}-k^{2})^{2}}-h^{2}}\end{aligned}.\]
Clearly the functions $\xl$ and $\xr$ are smooth for $k\in
(k_0(s),\beta)$.  It is straightforward to see that $k\mapsto \xr(k)$
is a strictly increasing smooth function for $k\in (k_0(s),\B)$.

We prove that the function $\xl$ is strictly decreasing by showing
$\xl'$ is always negative.  A somewhat tedious calculation shows that
\[\xl'(k) = \frac{\B^2 s (\B^2+1)k}{(\B^2-k^2)^2}
\bparen{2-\frac{\frac{(\B^2+1)s(\B^2+k^2)}{\B^2-k^2}}{\sqrt{\frac{\B^2s^2k^2(\B^2+1)^2}{(\B^2-k^2)^2}-h^2}}}
\]
By replacing the square root in this expression by the upper bound
$\frac{\B s k(\B^2+1)}{(\B^2-k^2)}$, we see that
\[\xl'(k) \leq \frac{\B^2 s (\B^2+1)k}{(\B^2-k^2)^2\B
k}(-1)(\B-k)^2\] and the right-hand side of this expression is clearly
negative.

The circles $C(s,\cdot)$ are symmetric about the $x_1$-axis, so if $j$
and $k$ are points in $(k_0(s),\B)$ with $j<k$, since
$\xl(k)<\xl(j)<\xr(j)<\xr(k)$, the circle $C(s,j)$ is strictly inside
the circle $C(s,k)$.  This proves \eqref{contain}.

By the choice of $\eps$ in \eqref{eps estimate}, $1-6\eps>k_1$ and
$1+6\eps<\B$.  Because $\xl(k)$ and $\xr(k)$ are smooth strictly
monotonic functions with nonzero derivatives, $(1-6\eps,1+6\eps)\ni
k\mapsto C(s,k)$ is a foliation of an open, connected region
containing $C(s,1)=\Stwox(s)$, and this proves \eqref{foliate}.
\end{proof}

We define
\Beq\label{def:psi2}\psi_2(s,x) =\begin{cases} 1&\abs{\frac{\B B}{A}-1
}<\eps\\
0&\abs{\frac{\B B}{A}-1 }>2\eps\end{cases}\Eeq and we extend smoothly
between (which is possible by Lemma \ref{lemma:circles}, statement
\eqref{foliate}).  By the lemma, $\psi_2(s,\cdot)$ is equal to $1$ on
an open neighborhood of $\Stwox(s)$ and zero away from $\Stwox(s)$.

We assume, without loss of generality, that $\psi_1$ and $\psi_2$ are
symmetric about the $x_1$-axis.

\begin{remark}\label{rem:g(s,t)}  We now can  define the
function $g(s,t)$ in Remark \ref{S1S2cutoff}.  We let
\bel{def:D}D(s,\eps)= \sparen{(x_1,x_2)\st |x_2|<\eps, \abs{\frac{\B
B(s,x)}{A(s,x)}-1}<\eps}.\ee The set $D(s,4\eps)$ is represented by
the shaded set in Figure \ref{fig:CutoffsGood} that is near
$C(s,1)=\Stwox(s)$ and the $x_1$-axis.  Let $g$ be a smooth function
of $(s,t)$ that is zero if the ellipse $E(s,t)$ given in \eqref{def:E}
intersects $D(s,4\eps)$ and is equal to $1$ if $E(s,t)$ does not meet
$D(s,5\eps)$.
\begin{figure}[h!]
\begin{center}\includegraphics[width=4in]{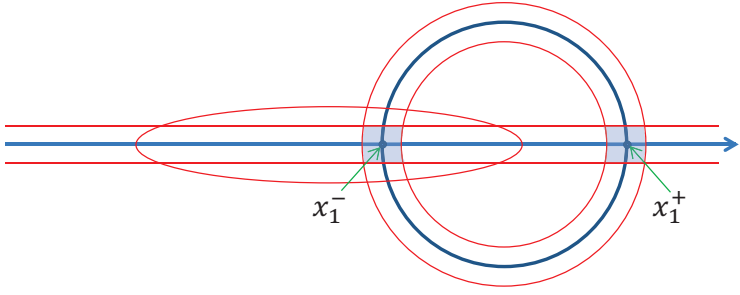}
\end{center}\caption{Picture of ellipse   $E(s,t)$ that does not
meet $D(s,4\eps)$.  As discussed in Remark \ref{S1S2cutoff} and Remark
\ref{rem:g(s,t)}, ellipses are muted by $g$ if they intersect
$D(s,4\eps)$.}\label{fig:CutoffsGood}
\end{figure}
\end{remark}

\subsubsection{Properties of $\Gc^*\Gc$ and end of proof}
 We now write
$\Gc=\Gc_0+\Gc_1+\Gc_2+\Gc_3$ where $\Gc_i$ are given in terms of
their kernels
\begin{align*}
&K_{\Gc_0}=\int e^{-i\varphi} a \psi_1 \psi_2 d \omega,\quad
K_{\Gc_1}=\int e^{-i\varphi} a \psi_1 (1-\psi_2) d \omega, \\
&K_{\Gc_2}=\int e^{-i\varphi} a(1-\psi_1) \psi_2 d \omega,\quad
K_{\Gc_3}=\int e^{-i\varphi} a (1-\psi_1) (1-\psi_2) d\omega,
\end{align*} where $\varphi$ is the phase function of $\Gc$.
The supports of the $\Gc_i$ are given in Figure
\ref{fig:supports-orig}.
\begin{figure}[h!]
\begin{center}\includegraphics[width=4in]{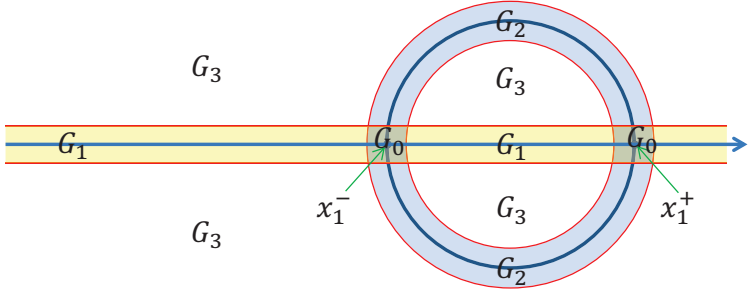}
 \end{center}\caption{Picture indicating the rough locations of
the support of $\Gc_0$, $\Gc_1$, $\Gc_2$, $\Gc_3$. Note that the
circles are not exactly concentric.}\label{fig:supports-orig}
\end{figure}

Now we consider $\Gc^*\Gc$, which using the decomposition of $\Gc$ as
above can be written as $\Gc^*\Gc =\Gc^*_0
\Gc+(\Gc_{1}+\Gc_{2})^{*}\Gc_{0}+ \Gc^*_1\Gc_1+
\Gc^*_2\Gc_2+\Gc^*_1\Gc_2+\Gc^*_2\Gc_1+\Gc^*_1\Gc_3+\Gc^*_2\Gc_3
+\Gc^*_3\Gc$\\

The theorem now follows from Lemmas \ref{ImagingSection:Lemma
G0}-\ref{ImagingSection:Lemma G1}, and Theorem
\ref{ImagingSection:I(p,l) regularity}, which we now state and prove.
In the lemmas, we analyze  the compositions above.

Recall that $\Gc_{1}$ and $\Gc_{2}$ are operators defined as follows:
 \[
 \Gc_{1}V(s,t)=\int e^{-\I \vp(s,t,x,\omega)} \psi_{1}(x)
 (1-\psi_{2}(s,x)) a(s,t,x,\omega) V(x) \D x \D \omega
 \]
 and
 \[
 \Gc_{2}V(s,t)=\int e^{-\I \vp(s,t,x,\omega)} (1-\psi_{1}(x))
 \psi_{2}(s,x) a(s,t,x,\omega) V(x) \D x \D \omega
 \]

\begin{lemma}\label{ImagingSection:Lemma G0}
The operators $\Gc_{1}^{*}\Gc_{2}$ and $\Gc_{2}^{*}\Gc_{1}$ are smoothing.
\end{lemma}
\begin{proof}
 We show that $\Gc_{1}^{*}\Gc_{2}$ is smoothing. The proof for the case of $\Gc_{2}^{*}\Gc_{1}$ is similar.

We have
 \[
 \Gc_{1}^{*} V(x)=\int e^{\I \vp(s,t,x,\omega)} \psi_{1}(x)(1-\psi_{2}(s,x))\overline{a(s,t,x,\omega)} V(s,t) \D s \D t \D \omega.
 \]
where  $\psi_1(x)$ and $\psi_2$ are defined in \eqref{def:psi1} and \eqref{def:psi2} respectively.
 The Schwartz kernel of $\Gc_{1}^{*} \Gc_{2}$ is
 \[ K(x,y)=\int e^{i\omega(
|y-\gamma_{T}(s)|+|y-\gamma_{R}(s)|-|x-\gamma_{T}(s)|-|x-\gamma_{R}(s)|)}\tilde{a}(x,y,s,\omega)\:
d s d \omega, \]
where $\wt{a}(x,y,s,\omega)$ has the following products of cutoff functions as an additional factor:
\[
g(s,t)\psi_{1}(x)(1-\psi_{2}(s,x))(1-\psi_{1}(y))\psi_{2}(y,s).
\]
Here $t$ is determined from $s$ and $x$ as the value for which
$x\in E(s,t)$. For this reason, in trying to understand the
propagation of singularities, we need only to restrict ourselves, for
each fixed $s\geq s_1$, to those base points $x$ and $y$ for which
\bel{supp}x \in \supp(\psi_{1}(\cdot)(1-\psi_{2}(s,\cdot)),\ y \in
\supp((1-\psi_{1}(\cdot)))\psi_{2}(s,\cdot)).\ee

We use this setup to show $\Gc_1^* \Gc_2$ is smoothing by showing its
symbol is zero for covectors in $ \CG^t\circ \CG$ (note that our
argument shows that the symbol of the operator is zero in a
neighborhood in $\paren{T^*(X)\smo}^2$ of $\CG^t\circ \CG$).  Let
$(x,\xi,y,\xi')\in \CG^t\circ \CG$.  Then, there is an
$(s_2,t_2,\eta)\in T^*(Y)\smo$ such that $(x,\xi,s_2,t_2,\eta)\in
\CG^t$ and $(s_2,t_2,\eta,y,\xi')\in \CG$.  For the rest of the proof,
we fix this $s_2$.  (If there are other values of $s$ associated to
the composition, we repeat this proof for those values of $s$.)

Because $\CG^t\circ \CG\subset \Delta\cup C_1\cup C_2$, we consider
three cases separately.

\begin{enumerate}[I.]
\item\label{item:Delta-case} \textbf{Covectors
$(x,\xi,y,\xi')\in\Delta\cap \paren{\CG^t\circ \CG}$:} In this case,
$x=y$ and $x$ is in $\supp(\psi_1)\cap \supp(\psi_2)\subset
D(s_2,4\eps)$.  By the choice of the function $g(s,t)$ in Remark
\ref{rem:g(s,t)}, the symbol of $\Gotc$ is zero above such $x$.

\item\label{item:C1-case} \textbf{Covectors $(x,\xi,y,\xi')\in
C_1\cap \paren{\CG^t\circ \CG}$:} In this case, $(x_1,x_2)=(y_1,-y_2)$
and the argument in case \ref{item:Delta-case} shows that the symbol
of $\Gotc$ is zero for such $x$ and $y$

\item\label{item:C2-case}\textbf{Covectors $ (x,\xi,y,\xi')\in
C_2\cap\paren{\CG^t\circ \CG}$:} If $(x,\xi,y,\xi')\in C_2\cap
\paren{\CG^t\circ \CG}$, then for some $(s_2,t_2)$ above, there is a
$k_2>k_0(s_2)$, such that
\begin{align}
x\in&E(s_2,t_2)\cap \supp(\psi_1)\cap\supp(1-\psi_2(s,\cdot))\cap
C(s_2,k_2)\label{x restriction}\\
y\in&E(s_2,t_2)\cap \supp(1-\psi_1)\cap\supp(\psi_2(s,\cdot))\cap
C\paren{s_2,1/k_2}.\label{y restriction}\end{align}

Using \eqref{x restriction}, the fact that $k_2 =\B
B(s_2,x)/A(s_2,x)$, we see that $
\abs{x_2}<2\eps$ and $\abs{1-k_2}>\eps$.  Now, using the restriction
on $1/k_2$ in \eqref{y restriction} and the fact that $\eps<1/4$, we
see $\abs{1-k_2}<4\eps$.  Putting this together shows that
\[ 1-4\eps<k=\frac{\B B(s_2,x)}{A(s_2,x)}<1+4\eps.\]
Since $\abs{x_2}<2\eps$, this shows that $x\in D(s_2,4\eps)$.
Therefore $E(s_2,t_2)\cap D(s_2,4\eps)\neq \emptyset$ and
$g(s_2,t_2)=0$ by Remark \ref{rem:g(s,t)}.  Therefore, the symbol of
$\Gotc$ is zero near $(x,\xi,y,\xi')$ so $\Gotc$ is smoothing near
$(x,\xi,y,\xi')$.
\end{enumerate}

This finishes the proof that $\Gotc$ is smoothing.
\end{proof}

\begin{lemma}
The operator $\Gc_{0}$ is smoothing.
\end{lemma}
\bpr
Recall that the Schwartz kernel of $\Gc_0$ is
\[K_{\Gc_0}= \int e^{-i\vp} a\psi_1(x)\psi_2(s,x)d\omega. \]
For each fixed $s\geq s_1$, the support of
$\psi_1(\cdot)\psi_2(s,\cdot)$ is inside $D(s,4\eps)$ and by the
choice of the function $g(s,t)$ in Remark \ref{rem:g(s,t)}, the symbol
of $\Gc_0$ is zero above such $(s,x)$.  \epr

 \begin{lemma}\label{ImagingSection:Lemma G1} The operators
$\Gc^*_1\Gc_3$, $\Gc^*_2 \Gc_3$ and $\Gc^*_3\Gc$ can be decomposed as
a sum of operators belonging to the space $I^{3}(\Delta\setminus
(C_1\cup C_2))+I^{3}(C_{1}\setminus (\Delta\cup C_2))+I^{3}(C_2
\setminus (\Delta\cup C_1))$.
\end{lemma}

\begin{proof} Each of these compositions is covered by the transverse
intersection calculus.

We decompose $\Gc_1$, $\Gc_2$, and $\Gc_3$ into a sum of operators on
which the compositions will be easier to analyze.  This is
represented in Figure \ref{fig:supports}.
\begin{figure}[h!]
\begin{center}\includegraphics[width=4in]{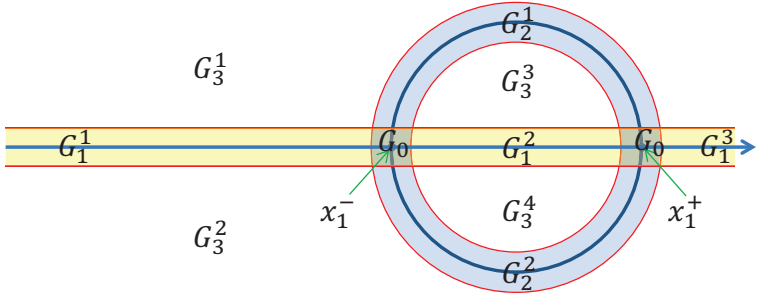}
 \end{center}\caption{Picture indicating the rough locations of the
support of $\Gc_0$, and the partitioned functions $\Gc_1^i$,
$\Gc_2^j$, and $\Gc_3^k$. Note that the circles are not exactly
concentric.}\label{fig:supports}
\end{figure}

For $\Gc_1$, note that $\Stwox(s)$ divides $\sparen{\abs{x_2}<2\eps}$
in three regions since $r(s,1)>12\eps$ by \eqref{eps condition}.  Let
$H_1(s)$ be the part of $\sparen{\abs{x_2}<2\eps}\setminus \Stwox(s)$
to the left of $\Stwox(s)$ and let $H_2(s)$ be the part inside
$\Stwox(s)$ and $H_3(s)$ the part to the right of $\Stwox(s)$.  Define
our partitioned operators as follows
$\Gc_{1}=\Gc_{1}^{1}+\Gc_{1}^{2}+\Gc_{1}^{3}$ where
{\small{\begin{align*} \Gc_{1}^{i}V(s,t)=\int e^{-i
\varphi(s,t,x,\omega)}\psi_{1}(x)(1-\psi_{2}(s,x))\chi_{H_i(s)}(x)a(s,t,x,\omega)V(x)d
x d \omega
\end{align*}}}
for $i=1,2,3$.  Note that the symbols are all smooth because
\[\chi_{H_j(s)}(x)\psi_1(x)\paren{1-\psi_2(s,x)}\] is a smooth cutoff
function in $(s,x)$ since the support of $\psi_1$ is inside
$\sparen{\abs{x_2}<2\eps}$ and the support of $(1-\psi_2(s,\cdot))$
does not meet $\Stwo(s)$.

We decompose $\Gc_2$ into two operators in a similar way.  Let $I_1$
be the open upper half plane and let $I_2$ be the open lower half
plane.  Define \bel{def:G2i} \Gc_{2}^{i}V(s,t)=\int e^{-i
\varphi(s,t,x,\omega)}(1-\psi_{1}(x))\psi_{2}(s,x)\chi_{I_j(s)}(x)a(s,t,x,\omega)V(x)d
x d \omega
\ee for $j=1,2$
Because  the functions
$(1-\psi_1(\cdot))\psi_2(s,\cdot)$ are supported away from the $x_1$
axis, these symbols are smooth.

We decompose $\Gc_3$ into four operators in a similar way using Figure
\ref{fig:supports}: $\Stwox(s)$ divides $\sparen{x_2\neq 0}$ into four
regions $J_1(s)$, the unbounded region above the $x_1$-axis,
$J_2(s)$, it's mirror image in the $x_1$-axis, $J_3(s),$ the bounded
region inside $\Stwox(s)$ and above the $x_1$-axis, and its mirror
image, $J_4(s)$.  We define
{\small{\begin{align*} \Gc_{3}^{k}V(s,t)=\int e^{-i
\varphi(s,t,x,\omega)}(1-\psi_{1}(x))(1-\psi_{2}(s,x))\chi_{J_k(s)}(x)a(s,t,x,\omega)V(x)d
x d \omega
\end{align*}}} for $k=1,2,3,4$, and because of the cutoffs used, these
are all FIO with smooth symbols.

To find the canonical relation of ${\Gc_1^j}^*\Gc_3^k$, we consider
$(x,\xi,y,\xi')\in \CG^t\circ \CG$ and let $(s,t)\in Y$ such that
$(x,\xi,s,t,\eta)\in \CG^t$ and $(s,t,\eta,y,\xi')\in \CG$.  In any
case, $(\Gc_1^i)^*\Gc_3^j$ has canonical relation a subset of
$\CG^t\circ \CG \subset \Delta\cup C_1\cup C_2$. To find which subset,
we consider the restriction that the supports of the $\Gc_i^j$ put on
$x$ and $y$.  We use the fact that $x$ and $y$ are on $E(s,t)$ plus
the following rules to understand the canonical relations of these
operators:
\begin{enumerate}[$\left(i\right)$]

\item\label{Delta canonical} If the supports exclude $x$ and $y$ from being
  equal, then the canonical relation ($WF'$) of the composed operator does
  not include $\Delta$.

\item\label{C1 canonical} if the supports exclude $x$ and $y$ from being reflections in
the $x_1$ axis then the canonical relation of the composed
operator does not include $C_1$.

\item\label{C2 canonical} If the supports exclude $x$ from being
outside $\Stwox(s)$ and $y$ being inside or vice versa, then the
canonical relation of the composed operator does not include $C_2$.

\end{enumerate}

%THESE ARE PROVIDED HERE SO YOU CAN SEE WHAT THESE WEIRD DEFINITIONS
%are.  They are in main.tex

% \newcommand{\Go}{\mathcal{G}_1}
% \newcommand{\Goo}{{\mathcal{G}_1^1}}
% \newcommand{\Got}{{\mathcal{G}_1^2}}
% \newcommand{\Goth}{{\mathcal{G}_1^3}}
%
% \newcommand{\Goop}{{\mathcal{G}_1^{1+}}}
% \newcommand{\Gotp}{{\mathcal{G}_1^{2+}}}
% \newcommand{\Gothp}{{\mathcal{G}_1^{3+}}}
%
% \newcommand{\Gooo}{{\mathcal{G}_1^{1o}}}
% \newcommand{\Goto}{{\mathcal{G}_1^{2o}}}
% \newcommand{\Gotho}{{\mathcal{G}_1^{3o}}}
%
% \newcommand{\Goom}{{\mathcal{G}_1^{1-}}}
% \newcommand{\Gotm}{{\mathcal{G}_1^{2-}}}
% \newcommand{\Gothm}{{\mathcal{G}_1^{3-}}}
%
%
% \newcommand{\Gt}{\mathcal{G}_2}
% \newcommand{\Gto}{{\mathcal{G}_2^1}}
% \newcommand{\Gtt}{{\mathcal{G}_2^2}}
% \newcommand{\Gth}{\mathcal{G}_3}
% \newcommand{\Gtho}{{\mathcal{G}_3^1}}
% \newcommand{\Gtht}{{\mathcal{G}_3^2}}
% \newcommand{\Gthth}{{\mathcal{G}_3^3}}
% \newcommand{\Gthf}{{\mathcal{G}_3^4}}
%
% \newcommand{\sgp}{\operatorname{\sigma^+}}
% \newcommand{\sgo}{\operatorname{\sigma^o}}
% \newcommand{\sgm}{\operatorname{\sigma^-}}

We first consider $(\Go)^*\Gth$.  To do this, we partition $\Go$
further.  Let $u$ be a smooth cutoff function supported in
$[-\eps,\eps]$ and equal to one on $[-\eps/2,\eps/2]$ and let $\sgp =
\chi_{[0,2\eps]}(1-u)\psi_1(1-\psi_2)$, $\sigma^o=
\chi_{[-\eps,\eps]}u \psi_1(1-\psi_2)$, and $\sgm =
\chi_{[-2\eps,0]}(1-u)\psi_1(1-\psi_2)$ where the characteristic
functions and $u$ are functions of $x_2$.  Note that, for each fixed
$s$ and functions of $x$, $\supp(\sgp)\subset\rr\times
[\eps/2,2\eps]$, $\supp(\sgo)\subset\rr\times [-\eps,\eps]$,
$\supp(\sgm)\subset\rr\times [-2\eps,-\eps/2]$.  All these functions
are smooth and $\psi_1(1-\psi_2)=\sgp+\sgo+\sgm$.  This allows us to
divide up each $\Go^j$ ($j=1,2,3)$ into the sum of three operators
where $\Go^{j+}(V)$ has symbol equal to the symbol of $\Gc$ but
multiplied by $\sgp H_j $, $\Go^{jo}(V)$ has symbol equal to the
symbol of $\Gc$ but multiplied by $\sgo H_j $, and $\Go^{j-}(V)$ has
symbol equal to the symbol of $\Gc$ but multiplied by $\sgm H_j $.
Note that $\Go^j = \Go^{j+}+\Go^{jo}+\Go^{j-}$.

 We now analyze the composition $\Goo^*\Gtho$ using this partition of
$\Goo$.  Consider the composition $(\Goop)^*\Gtho$.  Because both
operators are supported in $x$ above the $x_1$ axis, the canonical
relation of this composition cannot intersect $C_1$ (see ({\it\ref{C1
canonical}})).  Because they are both supported outside $\Stwox(s)$, it
cannot intersect $C_2$ (since $C_2$ associates points inside
$\Stwox(s)$ only with points outside $\Stwox(s)$ and vice versa by
({\it \ref{C2 canonical}})).  So this shows $(\Goop)^*\Gtho\in
I(\Delta\setminus C_1)$.

Note that we use the transverse intersection calculus to show
$(\Goop)^* \Gtho$ and each of the other operators in this lemma are
regular FIO.

Now, we consider $(\Gooo)^*\Gtho$.  Note that $\Gooo$ is supported in
$x$ in $\abs{x_2}<\eps$ and $\Gtho$ is supported in $x_2>\eps$.
Therefore, the canonical relation of the composition can include
neither $\Delta$ nor $C_1$ by ({\it \ref{Delta canonical}}), ({\it \ref{C1
canonical}}).  Furthermore, because they are both supported outside
$\Stwox(s)$, it does not contain $C_2$ by ({\it \ref{C2 canonical}}).
Therefore, $(\Gooo)^*\Gtho$ is smoothing.

Next, we consider $(\Goom)^*\Gtho$.  The argument is similar to the
case $(\Goop)^*\Gtho$, but this canonical relation is contained in
$C_1\setminus \Delta$.

This shows that $(\Goo)^*\Gtho$ is a sum of operators in
$I^3(\Delta\setminus (C_1\cup C_2))+I^3(C_1\setminus (\Delta\cup
C_2))$.

The proof that $(\Gc_{1}^{1})^{*} \Gc_{3}^{2} \in I^3(\Delta\setminus
(C_1\cup C_2))+I^3(C_1\setminus (\Delta\cup C_2))$ follows using the same arguments but
the roles of $\Goom$ and $\Goop$ are switched because $\Gtht$ has
support in $x$ below the $x_1$-axis and below $\Stwox(s)$.

Now we consider $(\Goo)^*\Gthth$.  Because the support in $x$ of
$\Goo$ is to the left of $\Stwox(s)$ and the support of $\Gthth$ is
inside, the canonical relation of $(\Goo)^*\Gthth$ cannot intersect
$\Delta$ (since there are no points $(x,\xi,x,\xi)$ in that canonical
relation by the support condition and ({\it \ref{Delta canonical}})
and it cannot intersect $C_1$ for a similar reason by ({\it \ref{C1
canonical}}).  So $(\Goo)^*\Gthth\in I^3(C_2\setminus (\Delta \cup
C_1))$.

A similar argument using symmetry of support of $\Gthth$ and $\Gthf$
in the $x_1$ axis shows that $(\Goo)^*\Gthf\in I^3(C_2\setminus (\Delta \cup
C_1))$.

Putting these together, we see that $(\Goo)^* \Gth\in I^3(\Delta
\setminus (C_1\cup C_2))+I^3(C_1\setminus (\Delta\cup
C_2))+I^3(C_2\setminus (\Delta\cup C_1))$.

The proof that  $(\Got)^*\Gth\in I^3(\Delta \setminus
(C_1\cup C_2))+I^3(C_1\setminus (\Delta\cup C_2))+I^3(C_2\setminus
(\Delta \cup C_1))$ is similar but here we
use the partition of $\Got$: $\Gotp$, $\Goto$ and $\Gotm$.
In a similar way, $(\Goth)^*\Gth\in I^3(\Delta\setminus
(C_1\cup C_2))+I^3(C_1\setminus (\Delta\cup C_2))+I^3(C_2\setminus
(\Delta\cup C_1))$.

Thus, $(\Go)^*\Gth\in  I^3(\Delta\setminus
(C_1\cup C_2))+I^3(C_1\setminus (\Delta\cup
C_2)+I^3(C_2\setminus(\Delta\cup C_1))$.

Now we consider $(\Gt)^*\Gth$.  Here we partition $\Gc_2^j$, $j=1,2$
into three operators with smooth symbols as we did for $\Go$:
\begin{itemize}
\item $\Gc_2^{j+}$ will have  support in $x$ for fixed $s$ in the
union of circles \hfil\newline $\cup_{k\in
[1+\eps,1+2\eps]} C(s,k)$ (outside of $\Stwox(s)$),

\item $\Gc_2^{jo}$ will have support in $x$ for fixed $s$ in the union
of circles $\cup_{k\in [1-\eps,1+\eps]} C(s,k)$ (surrounding
$\Stwox(s)$) and be equal to the symbol of $\Gt$ in \hfil\newline$\cup_{k\in
[1-\eps/2,1+\eps/2]} C(s,k)$, and

\item $\Gc_2^{j-}$ will have support
in $x$ for fixed $s$ in the union of circles \hfil\newline $\cup_{k\in
[1-2\eps,1-\eps]} C(s,k)$ (inside $\Stwox(s)$).
\end{itemize}
The proof follows similar arguments as for $(\Go)^*\Gth$ and it shows
$(\Gt)^*\Gth\in  I^3(\Delta\setminus
(C_1\cup C_2))+I^3(C_1\setminus (\Delta\cup
C_2)+I^3(C_2\setminus(\Delta\cup C_1))$.

Finally, we consider $(\Gth)^*\Gc$.  By symmetry of the conditions
({\it \ref{Delta canonical}}), ({\it \ref{C1 canonical}}), ({\it
\ref{C2 canonical}}), we justify $(\Gth)^*\Go$ and $(\Gth)^*\Gt$ are
in $ I^3(\Delta\setminus (C_1\cup C_2))+I^3(C_1\setminus (\Delta\cup
C_2))+I^3(C_2\setminus (\Delta \cup C_1))$.  So,
the only composition to consider is $(\Gth)^*\Gth$, and by analyzing
all combinations, we see $(\Gth)^*\Gth\in I^3(\Delta\setminus
(C_1\cup C_2))+I^3(C_1\setminus (\Delta\cup
C_2))+I^3(C_2\setminus(\Delta\cup C_1))$.  This finishes the proof.

\end{proof}

 We are left with the analysis of the compositions
$\Gc_{1}^{*}\Gc_{1}$ and $\Gc_{2}^{*}\Gc_{2}$. This is the content of
the next theorem:

\begin{theorem}
\label{ImagingSection:I(p,l) regularity} Let $\Gc_{1}$ and $\Gc_{2}$
be as above. Then
\begin{enumerate}[(a)]
\item \label{G1*G1} $\Gc_{1}^{*}\Gc_{1}\in I^{3,0}(\Delta,
C_{1})+I^3(C_2\setminus (\Delta \cup C_1))$.
\item\label{G2*G2} $\Gc_{2}^{*}\Gc_{2}\in I^{3,0}(\Delta,
C_2)+I^{3,0}(C_{1}, C_2)$.
\end{enumerate}
\end{theorem}

\bpr
 Consider the intersections of $\Delta, \ C_1, \ C_2$. We
have that $\Delta$ intersects $C_1$ cleanly in codimension $2$;
$\Delta$ intersects $C_2$ cleanly in codimension $1$ and $C_1$
intersects $C_2$ cleanly in codimension $2$.

For part \eqref{G1*G1} we decompose $\Go=\Goo+\Got+\Goth$.  Now, we
consider the compositions that $ (\Gc_{1}^{j})^{*}\Gc_{1}^{j} $ for
$j=1,2,3$.  Using \eqref{Delta canonical}, \eqref{C1 canonical}, and
\eqref{C2 canonical}, we have that
$WF'((\Gc_{1}^{j})^{*}\Gc_{1}^{j})\subset \Delta \cup C_{1}$.  Then,
using a proof similar to the one for Theorem \ref{thm:F*F alpha>=0},
we see that $(\Gc_{1}^{j})^{*}\Gc_{1}^{j}\in I^{3,0}(\Delta,C_{1})$.

 Arguments using \eqref{Delta canonical}, \eqref{C1 canonical}, and
\eqref{C2 canonical} show that the cross terms $(\Goo)^*\Got$,
$(\Got)^*\Goo$, $(\Got)^*\Goth$, and $(\Goth)^*\Got$ are in $I^3(C_2\setminus (\Delta \cup
C_1))$ and  $(\Goth)^*\Goo$ and $(\Goo)^*\Goth$ are smoothing.

\vskip .5 cm

Now, we consider part \eqref{G2*G2} and the operator $\Gc_2^*\Gc_2$.

We recall that $\Sigma_1$ and $\Sigma_2$ are disjoint, $\Sigma_2 \in C
\setminus \Sigma_1$ thus $C \setminus \Sigma_1$ is a two sided fold.
Next we use \cite{Melrose-Taylor} to get that $(C \setminus
\Sigma_1)^t \circ (C \setminus \Sigma_1)=\Delta \cup C_2$, and
that $C_2$ is a two sided fold.

\par We use the decomposition \eqref{def:G2i} $\Gc_2=\Gc_2^{1}+ \Gc_2^{2}$
where $\Gc_2^{1}$ is supported in the upper part of $\Sigma_2$ and
$\Gc_2^{2}$ is supported in the lower part of $\Sigma_2$.  Note that
the support in $x$ of $\Gc_2^1$ and $\Gc_2^2$ are disjoint.

Then using Theorem
\ref{thm:Felea Nolan} we have that \[(\Gc_2^{1})^*\Gc_2^{1} \in
I^{3,0}(\Delta, C_2)\ \ \text{and}\ \ (\Gc_2^{2})^*\Gc_2^{2} \in
I^{3,0}(\Delta, C_2).\]

Consider the operator $R$ defined as follows:
\[
RV(x_{1},x_{2})=V(x_{1},-x_{2}).
\]
This is a Fourier integral operator of order $0$ and it is easy to
check its canonical relation is $C_{1}$.
Let $\hat{\Gc}=\Gc_{2}^{2}\circ R$. We have $\hat{\Gc}^{*}\Gc_{2}^{1}\in I^{3,0}(\Delta,C_{2})$. Note that $C_1 \circ \Delta=C_1$, $C_{1}\circ
C_2=C_2$ and $C_1\times \Delta$ (as well as $C_1 \times
C_2$) intersects $T^{*}X\times \Delta_{T^{*}{X}}\times T^{*}X$
transversally. Using \cite[Proposition 4.1]{Guillemin-Uhlmann},  this implies that
$R^{*}\tilde{\Gc}^{*}\Gc_{2}^{1}\in I^{3,0}(C_1, C_2)$. Since
$\tilde{\Gc}^{*}=R^{*}(\Gc_{2}^{2})^{*}$ and $(R^{*})^{2}=\mbox{Id}$ we
have $(\Gc_{2}^{2})^{*}\Gc_{2}^{1}\in I^{3,0}(C_1, C_2)$

Similarly, we show that \[(\Gc_2^{1})^*\Gc_2^{2}
\in I^{3,0}(C_1, C_2).\]

This concludes the proof of Statement \eqref{O2} of Theorem
\ref{thm:G*G}.\epr

\subsection{Spotlighting}\label{sect:spotlighting}
This is equivalent to
assuming the scatterer $V$ has support in either the open half-plane $x_2>0$ or $x_2<0$.
In this case, $C_1$ does not appear in the analysis.

 \begin{theorem}\label{thm:beamforming} Let $\Gc$ be as in
\eqref{def:F} of order $\frac{3}{2}$. Assume the amplitude of $\Gc$ is
nonzero only on a subset of either the upper half-plane $(x_2>0)$ or
the lower half plane $x_{2}<0$ and bounded away from the $x_1$ axis.
Then $\Gc^*\Gc \in I^{3,0}(\Delta, C_2)$, where $C_2$ is given by
\eqref{def:C2}.

\end{theorem}

\begin{proof}
We assume $x_2> 0$ (the other case is similar), $\Sigma_1$ is empty and $\pi_L$ and $\pi_R$ have fold
singularities along $\Sigma_2$ as proved in Proposition \ref{Wavefront
of composition II}. Thus $ C^t \circ C= \Delta \cup C_2$ where $C_2$ is a two-sided fold. Using the results in Felea \cite{RF1} and Nolan
\cite{Nolan-fold_caustics}, we have that $\Gc^*\Gc \in I^{3,0}(\Delta, C_2)$.
 \end{proof}

In this case, $C_2$ does contribute to the added
singularities and this is discussed in Remark \ref{remark:alpha- general}.

\section{Acknowledgements}

All authors thank The American Institute of Mathematics (AIM) for the
SQuaREs (Structured Quartet Research Ensembles) award, which enabled
their research collaboration, and for the hospitality during the
authors' visits to AIM in 2011, 2012, and 2013.  Most of the results
in this paper were obtained during the last two visits. Support by the
Institut Mittag-Leffler (Djursholm, Sweden) is gratefully acknowledged
by Krishnan, Nolan, and Quinto.

The authors thank the referees for their thorough, thoughtful,
and insightful comments that made the article clearer.

Ambartsoumian was supported in part by NSF grants DMS 1109417 and DMS
1616564, and by Simons Foundation grant 360357.

Felea was supported in part by Simons Foundation grant 209850.

Krishnan was supported in part by NSF grants DMS 1109417 and DMS
1616564. He also benefited from the support of Airbus Corporate
Foundation Chair grant titled ``Mathematics of Complex Systems''
established at TIFR CAM and TIFR ICTS, Bangalore, India.

Quinto was partially supported by NSF grants DMS 1311558
 and DMS 1712207,
and a fellowship from the Otto M{\o}nsteds Fond during fall 2016 at
the Danish Technical University as well from the Tufts University
Faculty Research Awards Committee.

\begin{appendix}

\section{Proofs of iterated regularity for $\Fc$ ($\alpha \geq
0$)}\label{proofs:iterated regularity}

In this section, we prove that each of the $\tilde{p}_i$ given
in \eqref{generators:F} is a sum of products of derivatives of $\Phi$
and smooth functions. This will finish the proof that $\Fc^*\Fc\in
I^{3,0}(\Delta,C_1)$.

\subsection{Expression for $x_{1}-y_{1}$} We will use the same prolate
spheroidal coordinates (\ref{Prolate coordinate system}) with foci
$\gr(s)$ and $\gt(s)$ to solve for $x$ and $y$.  We have
\begin{align}
\notag x_{1}-y_{1}&= \lb \frac{1+\alpha}{2}s +
\frac{1-\alpha}{2}s\cosh \rho \cos \phi\rb  \\
\notag&\qquad- \lb \frac{1+\alpha}{2}s + \frac{1-\A}{2}s\cosh \rho' \cos \phi'\rb \\
\notag & = \frac{ 1-\A}{2} s\lb \cosh\rho \cos\phi - \cosh\rho' \cos \phi'\rb\\
\label{x1-y1 prolate coordinates expression}& = \frac{1-\A}{2} s\lb (\cosh \rho -\cosh\rho')\cos\phi + \cosh \rho'( \cos \phi-\cos\phi')\rb.
\end{align}
We have
\begin{align*}\PD_{\omega} \Phi& = \big(
 \norm{y-\gamma_{T}(s)}+\norm{y-\gamma_{R}(s)}
-(\norm{x-\gamma_{T}(s)}+\norm{x-\gamma_{R}(s)})\big)\\
&=(1-\A)s (\cosh\rho'-\cosh\rho).
\end{align*}
Therefore in \eqref{x1-y1 prolate coordinates expression}, it is enough to express $\cos \phi-\cos \phi'$ in terms of $\PD_{\omega} \Phi$ and $\PD_{s}\Phi$. We obtain:
\begin{align*}
\frac{\PD_{s}\Phi}{\omega}& = \lb \A \frac{x_{1}-\A s }{A} + \frac{x_{1}-s}{B}\rb - \lb \A \frac{y_{1}-\A s }{A'} + \frac{y_{1}-s}{B'}\rb \\
& =\A\frac{\cosh \rho \cos \phi+1}{\cosh \rho + \cos \phi} + \frac{\cosh \rho \cos \phi -1}{\cosh \rho -\cos \phi}\\
&\quad -\lb \A\frac{\cosh \rho' \cos \phi'+1}{\cosh \rho' + \cos
\phi'} + \frac{\cosh \rho' \cos \phi' -1}{\cosh \rho' -\cos
\phi'}\rb\end{align*}
Combining the first and the third term, and second and the fourth term above and then simplifying, we get
{\small \begin{align*}& = \A \frac{(\cos \phi -\cos \phi')(\cosh \rho
 \cosh\rho'-1)+ (\cosh \rho-\cosh\rho')(\cos\phi \cos\phi'-1)}{(\cosh
 \rho'+\cos\phi')(\cosh \rho +\cos\phi)}\\
 & \quad+ \frac{(\cos \phi - \cos \phi')(\cosh \rho \cosh \rho'-1)+ (\cosh \rho -\cosh\rho')(1-\cos \phi \cos \phi')}{(\cosh \rho -\cos\phi)(\cosh \rho'-\cos\phi')}\\
 & = (\cos \phi - \cos \phi')(\cosh \rho \cosh \rho'-1)\\
 & \quad\quad\times\lb \frac{\A}{(\cosh \rho'+\cos\phi')(\cosh \rho +\cos\phi)}+ \frac{1}{(\cosh\rho - \cos \phi)(\cosh\rho'- \cos\phi')}\rb\\
 & \quad+ (\cosh \rho -\cosh\rho')(\cos \phi \cos\phi'-1)\\
 & \quad \quad\times\lb \frac{\A}{(\cosh \rho'+\cos\phi')(\cosh \rho
 +\cos\phi)}-\frac{1}{(\cosh\rho - \cos \phi)(\cosh\rho'-
 \cos\phi')}\rb
 \end{align*}}
where $\times$ indicates multiplication with the
expression in the previous line. Now denote
\[
P_{\pm}:=\frac{\A}{(\cosh \rho'+\cos\phi')(\cosh \rho +\cos\phi)}\pm  \frac{1}{(\cosh\rho - \cos \phi)(\cosh\rho'- \cos\phi')}
\]
Note that since $\A>0$, $P_{+}>0$.
Therefore we have
\begin{align*}
\cos \phi - \cos \phi' = \frac{1}{ (\cosh\rho \cosh\rho'-1)P_{+}}\lb \frac{1}{\omega} \PD_{s}\Phi-\frac{(1-\cos \phi \cos \phi')}{(1-\A)s} P_{-} \PD_{\omega} \Phi\rb
\end{align*}
Now using this expression for the difference of cosines in \eqref{x1-y1 prolate coordinates expression}, we are done.
\subsection{Expression for $x_{2}^{2} - y_{2}^{2}$}
We have
\begin{align}
\notag x_{2}^{2} - y_{2}^{2} & = \frac{(1-\A)^{2} s^{2}}{4}\lb \sinh^{2} \rho \sin^{2} \phi \cos^{2} \theta - \sinh^{2} \rho' \sin^{2} \phi' \cos^{2} \theta'\rb\\
\notag & = \frac{(1-\A)^{2} s^{2}}{4}\lb \sinh^{2} \rho \sin^{2} \phi - \sinh^{2} \rho' \sin^{2} \phi'\rb \\
\label{0 term} & + \frac{(1-\A)^{2} s^{2}}{4}\lb  -\sinh^{2} \rho \sin^{2} \phi \sin^{2} \theta + \sinh^{2} \rho' \sin^{2} \phi' \sin^{2} \theta'\rb
\end{align}
Since $x_{3}=y_{3}=0$, we have that the last term in \eqref{0 term} is $0$.

Now we can write \\
$
\sinh^{2} \rho \sin^{2} \phi -\sinh^{2} \rho'\sin^{2} \phi' = (\cosh^{2} \rho -\cosh^{2} \rho')\sin^{2} \phi -(\cos^{2} \phi -\cos^{2} \phi')\sinh^{2} \rho'= \\ (\cosh\rho -\cosh\rho')(\cosh\rho +\cosh\rho')\sin^{2} \phi -(\cos \phi -\cos \phi') (\cos \phi +\cos \phi') \sinh^{2} \rho'.
$\\
Since $\cosh \rho -\cosh \rho'$  and $\cos \phi -\cos\phi'$ can be expressed in terms of $\PD_{\omega} \Phi$ and $\PD_{s}\Phi$ as above, we are done.

\subsection{Expression for $\xi_{1}-\eta_{1}$}
We have
\begin{align*}
& \xi_1=-\omega \left(\frac{x_1-\alpha s}{\sqrt{(x_1-\alpha
s)^2+x_2^2+h^2}}+\frac{x_1-s}{\sqrt{(x_1-s)^2+x_2^2+h^2}}\right)\\
&\eta_{1}=-\omega \left(\frac{y_1-\alpha s}{\sqrt{(y_1-\alpha
s)^2+y_2^2+h^2}}+\frac{y_1-s}{\sqrt{(y_1-s)^2+y_2^2+h^2}}\right)
\end{align*}
In prolate spheroidal coordinates, we have
\begin{align*}
\frac{\xi_{1} -\eta_{1}}{2\omega}& = \lb \frac{ \sinh^{2} \rho' \cos \phi'}{\cosh^{2} \rho'-\cos^{2} \phi'} - \frac{\sinh^{2} \rho \cos \phi}{\cosh^{2} \rho-\cos^{2} \phi}\rb \\
& = \Bigg{(}\frac{ \sinh^{2} \rho' \cos \phi'}{\cosh^{2} \rho'-\cos^{2} \phi'} - \frac{ \sinh^{2} \rho' \cos \phi'}{\cosh^{2} \rho-\cos^{2} \phi}\\
& \qquad +\frac{ \sinh^{2} \rho' \cos \phi'}{\cosh^{2} \rho-\cos^{2} \phi}-\frac{\sinh^{2} \rho \cos \phi}{\cosh^{2} \rho-\cos^{2} \phi}\Bigg{)}\\
&= \sinh^{2} \rho' \cos \phi'\Bigg{(}\frac{ \cosh^{2} \rho-\cosh^{2} \phi' + \cos^{2} \phi'-\cos^{2} \phi}{(\cosh^{2} \rho'-\cos^{2} \phi')(\cosh^{2} \rho-\cos^{2} \phi)}\Bigg{)} \\
& \qquad +\Bigg{(}\frac{ \sinh^{2} \rho' \cos \phi'-\sinh^{2} \rho \cos \phi}{\cosh^{2} \rho-\cos^{2} \phi}\Bigg{)}.\\
&=\sinh^{2} \rho' \cos \phi'\Bigg{(}\frac{ \cosh^{2} \rho-\cosh^{2} \phi' + \cos^{2} \phi'-\cos^{2} \phi}{(\cosh^{2} \rho'-\cos^{2} \phi')(\cosh^{2} \rho-\cos^{2} \phi)}\Bigg{)} \\
& \qquad +\Bigg{(}\frac{ (\cosh^{2} \rho'-\cosh^{2} \rho) \cos \phi'+\sinh^{2} \rho(\cos\phi'- \cos \phi)}{\cosh^{2} \rho-\cos^{2} \phi}\Bigg{)}.
\end{align*}
As before we get the terms $\cosh \rho -\cosh \rho'$  and $\cos \phi -\cos\phi'$ which can be expressed in terms of $\PD_{\omega} \Phi$ and $\PD_{s}\Phi$.

\subsection{Expression for $(x_{2}-y_{2})(\xi_{2} + \eta_{2})$}
We have \\
$\xi_2=-\omega \left(\frac{x_2}{\sqrt{(x_1-\alpha
s)^2+x_2^2+h^2}}+\frac{x_2}{\sqrt{(x_1-s)^2+x_2^2+h^2}}\right)$ \  and \\
$\eta_{2}=-\omega \left(\frac{y_2}{\sqrt{(y_1-\alpha
s)^2+y_2^2+h^2}}+\frac{y_2}{\sqrt{(y_1-s)^2+y_2^2+h^2}}\right)$ \\
Thus
\begin{align*}
&-\frac{(x_{2}-y_{2})(\xi_{2}+\eta_{2})}{\frac{4}{(1-\A)s}\omega} =\frac{x_{2}^{2}\cosh \rho}{\cosh^{2} \rho -\cos^{2} \phi} -\frac{y_{2}^{2} \cosh\rho'}{\cosh^{2} \rho'-\cos^{2}\phi'} \\
& \qquad\qquad\qquad +x_{2}y_{2} \lb \frac{\cosh \rho'}{\cosh^{2} \rho'-\cos^{2}\phi'}-\frac{\cosh \rho}{\cosh^{2} \rho -\cos^{2} \phi}\rb\\
& \qquad= (x_{2}^{2} -y_{2}^{2})\frac{\cosh \rho}{\cosh^{2} \rho -\cos^{2} \phi}\\
&\qquad\qquad\qquad  - y_{2}(x_{2}-y_{2}) \lb \frac{ \cosh \rho}{\cosh^{2} \rho -\cos^{2}\phi} - \frac{ \cosh \rho'}{\cosh^{2} \rho' -\cos^{2}\phi'}\rb.
\end{align*}
Now
\begin{align*}
\frac{ \cosh \rho}{\cosh^{2} \rho  -\cos^{2}\phi} & - \frac{ \cosh \rho'}{\cosh^{2} \rho' -\cos^{2}\phi'} = \frac{ \cosh \rho-\cosh \rho'}{\cosh^{2} \rho -\cos^{2}\phi}\\
&+ \cosh \rho' \lb \frac{\cosh^{2} \rho'-\cosh^{2} \rho + \cos^{2} \phi -\cos^{2} \phi'}{(\cosh^{2} \rho -\cos^{2} \phi)(\cosh^{2} \rho'-\cos^{2} \phi')}\rb.
\end{align*}
Next we use again the expressions for $\cosh \rho -\cosh \rho'$  and
$\cos \phi -\cos\phi'$ as before and for $x_2^2-y_2^2$ we use
\eqref{0 term}.

\subsection{Expression for $(x_{2}+y_{2})(\xi_{2} -\eta_{2})$}
We have
\begin{align*}
&\frac{(x_{2}+y_{2})(\xi_{2}-\eta_{2})}{\frac{4}{(1-\A)s}\omega}\\
&\qquad =\frac{-x_{2}^{2}\cosh \rho}{\cosh^{2} \rho -\cos^{2} \phi} +\frac{y_{2}^{2} \cosh\rho'}{\cosh^{2} \rho'-\cos^{2}\phi'} \\
& \qquad\qquad +x_{2}y_{2} \lb \frac{\cosh \rho'}{\cosh^{2} \rho'-\cos^{2}\phi'}-\frac{\cosh \rho}{\cosh^{2} \rho -\cos^{2} \phi}\rb\\
&\qquad = (y_{2}^{2} -x_{2}^{2})\frac{\cosh \rho}{\cosh^{2} \rho -\cos^{2} \phi}\\
&\qquad\qquad  +y_{2}(x_{2}+y_{2}) \lb \frac{ \cosh \rho}{\cosh^{2} \rho -\cos^{2}\phi} - \frac{ \cosh \rho'}{\cosh^{2} \rho' -\cos^{2}\phi'}\rb.
\end{align*}
Now we are in a similar situation as in the previous case.
\subsection{Expression for $\xi_{2}^{2} -\eta_{2}^{2}$} We have
\begin{align*}
&\frac{\xi_{2}^{2} -\eta_{2}^{2}}{(\frac{4\omega}{(1-\A)s})^{2}} = \lb \frac{x_{2}^{2} \cosh^{2} \rho}{\cosh^{2} \rho -\cos^{2} \phi}-\frac{y_{2}^{2} \cosh^{2}\rho'}{\cosh^{2} \rho'-\cos^{2}\phi'}\rb\\
&\qquad=\frac{(x_{2}^{2} -y_{2}^{2})\cosh^{2} \rho}{\cosh^{2} \rho-\cos^{2} \phi}+ y_{2}^{2} \lb \frac{\cosh^{2} \rho}{\cosh^{2} \rho -\cos^{2} \phi}-\frac{\cosh^{2} \rho'}{\cosh^{2} \rho' -\cos^{2} \phi'}\rb\\
&\qquad=\frac{(x_{2}^{2} -y_{2}^{2})\cosh^{2} \rho}{\cosh^{2} \rho-\cos^{2} \phi}+ y_{2}^{2} \Bigg{(} \frac{\cosh^{2} \rho-\cosh^{2} \rho'}{\cosh^{2} \rho -\cos^{2} \phi}+\\
&\qquad\qquad \qquad \cosh^{2}\rho'\frac{(\cosh^{2} \rho'-\cosh^{2}\rho) +(\cos^{2} \phi -\cos^{2} \phi')}{(\cosh^{2} \rho' -\cos^{2} \phi')(\cosh^{2} \rho -\cos^{2} \phi)}\Bigg{)}.
\end{align*}
This part is complete as well.

\section{Expressions for $t^{-}_{s}$ and
$t^{+}_s$}\label{min-max-times}
\begin{figure}[h]
\begin{center}
\includegraphics[width=15.8cm,height=6.4cm,keepaspectratio]{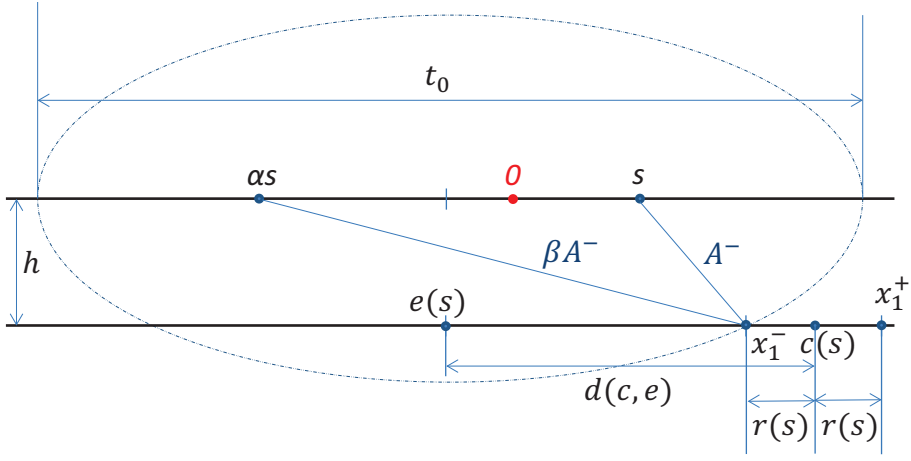}
\end{center}
\caption{The geometric setup of filtering, showing the vertical cross-section corresponding to $x_2=0$.}\label{filtering_pic}
\end{figure}
Recall that $\Sigma_2$ is defined in \eqref{def:Sigma2} as
\[\Sigma_2=\left\{(s,x,\omega)\in \CG:\;\left(x_1-\frac{2\alpha s}{\alpha+1}\right)^2 + x_2^2=-\alpha s^2\frac{(\alpha-1)^2}{(\alpha+1)^2}-h^2\right\}
\]
Recall that $s_0$ is defined by \eqref{def:s0} and for $s>\so$
$\Sigma_2$ is nonempty and not trivial.

We assume in this section that the cutoff function $f$ in
Section \ref{sect:main results} is chosen so it is
zero for $s\leq \so$.

The radius and the $x_1$-coordinate of the center of circle $\Sigma_2$ are
$$
r(s)=\sqrt{\frac{-\alpha s^2(\alpha-1)^2}{(\alpha+1)^2}-h^2},\;\;\mbox{and}\;\;
c(s)=\frac{2\alpha s}{\alpha+1}.
$$
Let $e(s)=(\alpha+1)s/2<0$ denote the $x_1$-coordinate of the center of ellipses in the plane. Then the distance between $e(s)$ and $c(s)$ can be written as
$$
d(c,e)=-\frac{s(\alpha-1)^2}{2(\alpha+1)}.%=\frac{s(\beta^2+1)^2}{2(\beta^2-1)}.
$$

For a fixed $s$ let $t^{-}_{s}$ and $t^{+}_{s}$ denote correspondingly the smallest and the largest values of $t$, for which the ellipsoid intersects $\Sigma_2$.
Notice, that since the normal to an ellipse at a point $P$ bisects the angle from the $P$ to the foci, the condition $\widetilde{\gamma_R}(s)\le\gamma_R(s)<c(s)$ implies that our ellipses on the ground can not intersect the circle $\Sigma_2$ at more than two points.
Here $\widetilde{\gamma_R}(s)$ denotes the right focus of the ellipse on the ground. Figure \ref{filtering_pic} shows the setup for $t_0$, where the ellipsoid passes through $x_1^-$, the closest to $e(s)$ point of $\Sigma_2$. The setup for $t^{+}_{s}$ is similar, with the ellipsoid passing through $x_1^+$, which is the farthest from $e(s)$ point of $\Sigma_2$.

A straightforward computation shows that
\bel{def:t0t1}
t^{-}_{s}=2\,(\beta+1)\sqrt{\frac{d\,(d-r)}{\beta^2+1}},\;\;\;t^{+}_{s}=2\,(\beta+1)\sqrt{\frac{d\,(d+r)}{\beta^2+1}},
\ee
where $\beta=\sqrt{-\alpha}$.

\end{appendix}

%GO TO siam PLAIN REFERENCE ONCE i CONVERT THE FILE TO sima STYLE

%\bibliographystyle{siamplain}

%\bibliography{SingularFIOS_II}

\begin{thebibliography}{10}

\bibitem{ABKQ}
{\sc G.~Ambartsoumian, J.~Boman, V.~P. Krishnan, and E.~T. Quinto}, {\em
  Microlocal analysis of an ultrasound transform with circular source and
  receiver trajectories}, in Geometric analysis and integral geometry, vol.~598
  of Contemp. Math., Amer. Math. Soc., Providence, RI, 2013, pp.~45--58,
  \url{https://doi.org/10.1090/conm/598/11983},
  \url{http://dx.doi.org/10.1090/conm/598/11983}.

\bibitem{AFKNQ:common_midpoint}
{\sc G.~Ambartsoumian, R.~Felea, V.~P. Krishnan, C.~Nolan, and E.~T. Quinto},
  {\em A class of singular {F}ourier integral operators in synthetic aperture
  radar imaging}, J. Funct. Anal., 264 (2013), pp.~246--269,
  \url{https://doi.org/10.1016/j.jfa.2012.10.008},
  \url{http://dx.doi.org/10.1016/j.jfa.2012.10.008}.

\bibitem{Amb-Kri}
{\sc G.~Ambartsoumian and V.~P. Krishnan}, {\em Inversion of a class of
  circular and elliptical {R}adon transforms}, in Complex analysis and
  dynamical systems {VI}. {P}art 1, vol.~653 of Contemp. Math., Amer. Math.
  Soc., Providence, RI, 2015, pp.~1--12,
  \url{https://doi.org/10.1090/conm/653/13174},
  \url{http://dx.doi.org/10.1090/conm/653/13174}.

\bibitem{Antoniano-Uhlmann}
{\sc J.~L. Antoniano and G.~A. Uhlmann}, {\em A functional calculus for a class
  of pseudodifferential operators with singular symbols}, in Pseudodifferential
  operators and applications ({N}otre {D}ame, {I}nd., 1984), vol.~43 of Proc.
  Sympos. Pure Math., Amer. Math. Soc., Providence, RI, 1985, pp.~5--16,
  \url{https://doi.org/10.1090/pspum/043/812280},
  \url{http://dx.doi.org/10.1090/pspum/043/812280}.

\bibitem{Beylkin-1985}
{\sc G.~Beylkin}, {\em Imaging of discontinuities in the inverse scattering
  problem by inversion of a causal generalized {R}adon transform}, J. Math.
  Phys., 26 (1985), pp.~99--108, \url{https://doi.org/10.1063/1.526755}.

\bibitem{Cheney2001}
{\sc M.~Cheney}, {\em A mathematical tutorial on synthetic aperture radar},
  SIAM Rev., 43 (2001), pp.~301--312 (electronic),
  \url{https://doi.org/10.1137/S0036144500368859},
  \url{http://dx.doi.org/10.1137/S0036144500368859}.

\bibitem{CheneyBordenbook}
{\sc M.~Cheney and B.~Borden}, {\em Fundamentals of Radar Imaging}, vol.~79 of
  {CBMS-NSF Regional Conference Series in Applied Mathematics}, {Society for
  Industrial and Applied Mathematics}, 2009.

\bibitem{Dehoop-InsideOut}
{\sc M.~V. de~Hoop}, {\em Microlocal analysis of seismic inverse scattering},
  in Inside out: inverse problems and applications, vol.~47 of Math. Sci. Res.
  Inst. Publ., Cambridge Univ. Press, Cambridge, 2003, pp.~219--296.

\bibitem{DuistermaatBook}
{\sc J.~J. Duistermaat}, {\em Fourier integral operators}, Modern Birkh\"auser
  Classics, Birkh\"auser/Springer, New York, 2011.
\newblock Reprint of the 1996 original.

\bibitem{RF1}
{\sc R.~Felea}, {\em Composition of {F}ourier integral operators with fold and
  blowdown singularities}, Comm. Partial Differential Equations, 30 (2005),
  pp.~1717--1740, \url{https://doi.org/10.1080/03605300500299968},
  \url{http://dx.doi.org/10.1080/03605300500299968}.

\bibitem{FG}
{\sc R.~Felea and A.~Greenleaf}, {\em {An FIO calculus for marine seismic
  imaging: folds and crosscaps}}, Communications in Partial Differential
  Equations, 33 (2008), pp.~45--77.

\bibitem{Felea-Greenleaf-MRL}
{\sc R.~Felea and A.~Greenleaf}, {\em Fourier integral operators with open
  umbrellas and seismic inversion for cusp caustics}, Math. Res. Lett., 17
  (2010), pp.~867--886, \url{https://doi.org/10.4310/MRL.2010.v17.n5.a6},
  \url{http://dx.doi.org/10.4310/MRL.2010.v17.n5.a6}.

\bibitem{Felea-Greenleaf-Pramanik}
{\sc R.~Felea, A.~Greenleaf, and M.~Pramanik}, {\em An {FIO} calculus for
  marine seismic imaging, {II}: {S}obolev estimates}, Math. Ann., 352 (2012),
  pp.~293--337, \url{https://doi.org/10.1007/s00208-011-0644-5},
  \url{http://dx.doi.org/10.1007/s00208-011-0644-5}.

\bibitem{Felea-Q2011}
{\sc R.~Felea and E.~T. Quinto}, {\em The microlocal properties of the local
  3-{D} {SPECT} operator}, SIAM J. Math. Anal., 43 (2011), pp.~1145--1157,
  \url{https://doi.org/10.1137/100807703},
  \url{http://dx.doi.org/10.1137/100807703}.

\bibitem{FLU}
{\sc D.~V. Finch, I.-R. Lan, and G.~Uhlmann}, {\em {Microlocal Analysis of the
  Restricted X-ray Transform with Sources on a Curve}}, in Inside Out, Inverse
  Problems and Applications, G.~Uhlmann, ed., vol.~47 of MSRI Publications,
  Cambridge University Press, 2003, pp.~193--218.

\bibitem{Frikel:2013gb}
{\sc J.~Frikel and E.~T. Quinto}, {\em {Characterization and reduction of
  artifacts in limited angle tomography}}, Inverse Problems, 29 (2013),
  p.~125007.

\bibitem{GU1989}
{\sc A.~Greenleaf and G.~Uhlmann}, {\em {Non-local inversion formulas for the
  X-ray transform}}, Duke Math. J., 58 (1989), pp.~205--240.

\bibitem{GU1990a}
{\sc A.~Greenleaf and G.~Uhlmann}, {\em Estimates for singular {R}adon
  transforms and pseudodifferential operators with singular symbols}, J. Funct.
  Anal., 89 (1990), pp.~202--232,
  \url{https://doi.org/10.1016/0022-1236(90)90011-9},
  \url{http://dx.doi.org/10.1016/0022-1236(90)90011-9}.

\bibitem{GU1990b}
{\sc A.~Greenleaf and G.~A. Uhlmann}, {\em Composition of some singular
  {F}ourier integral operators and estimates for restricted {X}-ray
  transforms}, Ann. Inst. Fourier (Grenoble), 40 (1990), pp.~443--466,
  \url{http://www.numdam.org/item?id=AIF_1990__40_2_443_0}.

\bibitem{Gu1985}
{\sc V.~Guillemin}, {\em {On some results of Gelfand in integral geometry}},
  Proceedings Symposia Pure Math., 43 (1985), pp.~149--155.

\bibitem{GuilleminSternberg}
{\sc V.~Guillemin and S.~Sternberg}, {\em Geometric asymptotics}, American
  Mathematical Society, Providence, R.I., 1977.
\newblock Mathematical Surveys, No. 14.

\bibitem{Guillemin-Uhlmann}
{\sc V.~Guillemin and G.~Uhlmann}, {\em Oscillatory integrals with singular
  symbols}, Duke Math. J., 48 (1981), pp.~251--267,
  \url{http://projecteuclid.org/getRecord?id=euclid.dmj/1077314493}.

\bibitem{Ho1971}
{\sc L.~H{\"o}rmander}, {\em Fourier integral operators. {I}}, Acta Math., 127
  (1971), pp.~79--183.

\bibitem{Katsevich:1997vo}
{\sc A.~I. Katsevich}, {\em {Local Tomography for the Limited-Angle Problem}},
  Journal of mathematical analysis and applications, 213 (1997), pp.~160--182.

\bibitem{Krishnan-Quinto}
{\sc V.~P. Krishnan and E.~T. Quinto}, {\em Microlocal aspects of bistatic
  synthetic aperture radar imaging}, Inverse Problems and Imaging, 5 (2011),
  pp.~659--674.

\bibitem{Melrose-Taylor}
{\sc R.~B. Melrose and M.~E. Taylor}, {\em Near peak scattering and the
  corrected {K}irchhoff approximation for a convex obstacle}, Adv. in Math., 55
  (1985), pp.~242--315, \url{https://doi.org/10.1016/0001-8708(85)90093-3},
  \url{http://dx.doi.org/10.1016/0001-8708(85)90093-3}.

\bibitem{MU}
{\sc R.~B. Melrose and G.~A. Uhlmann}, {\em Lagrangian intersection and the
  {C}auchy problem}, Comm. Pure Appl. Math., 32 (1979), pp.~483--519,
  \url{https://doi.org/10.1002/cpa.3160320403},
  \url{http://dx.doi.org/10.1002/cpa.3160320403}.

\bibitem{Moon-Heo}
{\sc S.~Moon and J.~Heo}, {\em Inversion of the elliptical {R}adon transform
  arising in migration imaging using the regular {R}adon transform}, J. Math.
  Anal. Appl., 436 (2016), pp.~138--148,
  \url{https://doi.org/10.1016/j.jmaa.2015.11.043},
  \url{http://dx.doi.org/10.1016/j.jmaa.2015.11.043}.

\bibitem{Nolan-fold_caustics}
{\sc C.~J. Nolan}, {\em Scattering in the presence of fold caustics}, SIAM J.
  Appl. Math., 61 (2000), pp.~659--672,
  \url{https://doi.org/10.1137/S0036139999356107},
  \url{http://dx.doi.org/10.1137/S0036139999356107}.

\bibitem{NC2004}
{\sc C.~J. Nolan and M.~Cheney}, {\em Microlocal analysis of synthetic aperture
  radar imaging}, J. Fourier Anal. Appl., 10 (2004), pp.~133--148,
  \url{https://doi.org/10.1007/s00041-004-8008-0},
  \url{http://dx.doi.org/10.1007/s00041-004-8008-0}.

\bibitem{NolSym1997}
{\sc C.~J. Nolan and W.~W. Symes}, {\em Global solution of a linearized inverse
  problem for the wave equation}, Comm. Partial Differential Equations, 22
  (1997), pp.~919--952, \url{https://doi.org/10.1080/03605309708821289}.

\bibitem{Quinto93}
{\sc E.~T. Quinto}, {\em Singularities of the {X}-ray transform and limited
  data tomography in {$\mathbb{R}^2$} and {$\mathbb{R}^3$}}, SIAM J. Math.
  Anal., 24 (1993), pp.~1215--1225.

\bibitem{QRS2011}
{\sc E.~T. Quinto, A.~Rieder, and T.~Schuster}, {\em {Local inversion of the
  sonar transform regularized by the approximate inverse}}, Inverse Problems,
  27 (2011), p.~035006 (18p),
  \url{https://doi.org/10.1088/0266-5611/27/3/035006}.

\bibitem{SU1}
{\sc P.~Stefanov and G.~Uhlmann}, {\em Stability estimates for the {X}-ray
  transform of tensor fields and boundary rigidity}, Duke Math. J., 123 (2004),
  pp.~445--467, \url{https://doi.org/10.1215/S0012-7094-04-12332-2},
  \url{http://dx.doi.org/10.1215/S0012-7094-04-12332-2}.

\bibitem{SU2}
{\sc P.~Stefanov and G.~Uhlmann}, {\em Boundary rigidity and stability for
  generic simple metrics}, J. Amer. Math. Soc., 18 (2005), pp.~975--1003
  (electronic), \url{https://doi.org/10.1090/S0894-0347-05-00494-7},
  \url{http://dx.doi.org/10.1090/S0894-0347-05-00494-7}.

\bibitem{Stefanov-Uhlmann-SAR}
{\sc P.~Stefanov and G.~Uhlmann}, {\em Is a curved flight path in {SAR} better
  than a straight one?}, SIAM J. Appl. Math., 73 (2013), pp.~1596--1612,
  \url{https://doi.org/10.1137/120882639},
  \url{http://dx.doi.org/10.1137/120882639}.

\bibitem{Stolk_deHoop_CPAM}
{\sc C.~C. Stolk and M.~V. de~Hoop}, {\em Microlocal analysis of seismic
  inverse scattering in anisotropic elastic media}, Comm. Pure Appl. Math., 55
  (2002), pp.~261--301, \url{https://doi.org/10.1002/cpa.10019},
  \url{http://dx.doi.org/10.1002/cpa.10019}.

\bibitem{tK-Smit-Verdel}
{\sc A.~P.~E. ten Kroode, D.-J. Smit, and A.~R. Verdel}, {\em A microlocal
  analysis of migration}, Wave Motion, 28 (1998), pp.~149--172,
  \url{https://doi.org/10.1016/S0165-2125(98)00004-3}.

\bibitem{Uhlmann-Vasy}
{\sc G.~Uhlmann and A.~Vasy}, {\em The inverse problem for the local geodesic
  ray transform}, Invent. Math., 205 (2016), pp.~83--120,
  \url{https://doi.org/10.1007/s00222-015-0631-7},
  \url{http://dx.doi.org/10.1007/s00222-015-0631-7}.

\bibitem{YCY}
{\sc B.~Yaz{\i}c{\i}, M.~Cheney, and C.~E. Yarman}, {\em Synthetic-aperture
  inversion in the presence of noise and clutter}, Inverse Problems, 22 (2006),
  pp.~1705--1729, \url{https://doi.org/10.1088/0266-5611/22/5/011},
  \url{http://dx.doi.org/10.1088/0266-5611/22/5/011}.

\end{thebibliography}
\def\dbar{\leavevmode\hbox to 0pt{\hskip.2ex \accent"16\hss}d}

\end{document}